\documentclass[11pt]{article}
\usepackage{amssymb,bm,amsmath, amsthm}
\usepackage{latexsym,natbib}

\usepackage{epsfig, psfrag}

\newtheorem{Lemma}{Lemma}

\newtheorem{Proposition}[Lemma]{Proposition}

\newtheorem{theorem}[Lemma]{Theorem}

\newtheorem{Condition}[Lemma]{Condition}
\newtheorem{Corollary}[Lemma]{Corollary}
\newtheorem*{Proposition*}{Proposition}

\newcommand{\ed}{\ \stackrel{d}{=} \ }

\newcommand{\bP}{{\bf P}}
\newcommand{\bE}{{\bf E}}

\newcommand{\bod}{{\bf d}}

\newcommand{\cE}{{\mathcal E}}

\def\gen{{\rm{gen}}}
\newcommand{\tA}{{\tilde{A}}}
\newcommand{\tB}{{\tilde{B}}}

\newcommand{\bxi}{{\overline{\xi}}}
\newcommand{\bK}{{\overline{K}}}
\newcommand{\bcC}{{\overline{\cC}}}

\newcommand{\Pb}{{\mathbb{P}}}
\newcommand{\Eb}{{\mathbb{E}}}

\newcommand{\Bi}{{\rm Bi}}

\def\bh{\overline{h}}

\def\bg{\bar{g}}
\def\bh{\bar{h}}
\def\tG{\tilde{G}}

\def\hz{\hat{z}}
\def\balpha{\boldsymbol{\alpha}}

\def\boldd{\boldsymbol{d}}
\def\boldk{\boldsymbol{k}}
\def\boldp{\boldsymbol{p}}
\def\boldt{\boldsymbol{t}}
\def\dm{d_{\rm max}}
\def\pto{\stackrel{p}{\to}}

\newcommand{\Nbold}{{\mathbb{N}}}

\def\ind{{\rm 1\hspace{-0.90ex}1}}

\newcommand{\cC}{\mathcal C}
\newcommand{\cI}{\mathcal I}

\title{Diffusion and Cascading Behavior\\ in Random Networks}
\author{Marc Lelarge\footnote{INRIA - ENS, 23 avenue d'Italie, 75013 Paris, France, tel:+33(0)1.39.63.55.33, fax:+33.(0)1.39.63.79.88,  Marc.Lelarge@ens.fr} }

\date{}
\begin{document}
\maketitle

\begin{abstract}
The spread of new ideas, behaviors or technologies has been extensively studied using epidemic models. Here we consider a model of diffusion where the
individuals' behavior is the result of a strategic choice. We study a simple coordination game with binary choice and give a condition for a new action to become
widespread in a random network. We also analyze the possible equilibria of this game and identify conditions for the coexistence of both strategies in large connected sets.
Finally we look at how can firms use social networks to promote their goals with limited information.
Our results differ strongly from the one derived with epidemic models and show that connectivity plays an ambiguous role: while it
allows the diffusion to spread, when the network is highly connected, the diffusion is also limited by high-degree nodes which are very stable. 
\end{abstract}



\thispagestyle{empty}

\newpage
\section{Introduction}
There is a vast literature on epidemics on complex networks (see
\citep{new03} for a review).
Most of the epidemic models consider a transmission mechanism which is
independent of the local condition faced by the agents concerned.
However, if there is a factor of coordination or persuasion involved, relative
considerations tend to be important in understanding whether some new
belief or behavior is adopted \citep{vega07}.
To fix ideas, it may be useful to think of the diffusion process as
modeling the adoption of a new technology. In this case, when an agent
is confronted with the possibility of adoption, her decision depends
on the persuasion effort exerted on her by each of her neighbors in
the social network.
More formally, those neighborhood effects can be captured as follows: the
probability that an agent adopts the new technology when $r$ out of her $d$
neighbors have already adopted can be modeled by a threshold function: this probability is one 
if $r/d\geq q$ and it is zero otherwise.
For $\theta=1/2$, this would correspond to a local majority rule.
More generally, some simple models for the diffusion of a new behavior
have been proposed in term of a basic underlying model of individual
decision-making: as individuals make decisions based on the choices of
their neighbors, a particular pattern of behavior can begin to spread
across the links of the network \citep{vega07}, \citep{easley2010}.
In this work, we analyze the diffusion in the large population limit when the underlying graph is a random network with given vertex degrees.
There is now a large literature on the study of complex networks
across different disciplines such as physics \citep{new03}, mathematics
\citep{rg}, sociology \citep{wat02}, networking \citep{sig08} or economics
\citep{vega07}. There is also a large literature on local
interaction and adoption externalities \citep{jy07}, \citep{lopez},
\citep{ggjvy08}, \citep{py09}. Similarly to these works, we study a game where players' payoffs depend on the actions taken by their neighbors in the network but not on the specific identities of these neighbors. Our most general framework (presented in Section \ref{sec:mod}) allows to deal with threshold games of complements, i.e. a player has an increasing incentive to take an action as more neighbors take the action.

The main contribution of our paper is a
model which allows us to study rigorously semi-anonymous threshold games of complements with local interactions on a complex network.
We use a very simple dynamics for the evolution of the play: each time, agents choose the strategy with best payoff, given the current behavior of their neighbors. Stochastic versions of best response dynamics have been studied extensively as a simple model for the emergence of conventions \citep{kmr}, \citep{you93}. A key finding in this line of work is that local interaction may drive the system towards a particular equilibrium in which all players take the same action. 
Most of the literature following \citep{el93}, \citep{bl93} is concerned with stochastic versions of best response dynamics on fixed networks and a simple condition known as risk dominance \citep{hs88} determines whether an innovation introduced in the network will eventually become widespread. In this case, bounds on the convergence time have been computed \citep{you00}, \citep{ms09} as a function of the structure of the interaction network.
In this paper, we focus on properties of deterministic (myopic) best response dynamics. In such a framework, as shown in \citep{bl95}, coordination on the risk dominant strategy depends on the network structure, the dynamics of the game and the initial configuration. In particular, this coordination might fail and equilibria with co-existent strategies can exist. This was shown in \citep{bl95} already in the simple case where the underlying network is the two-dimensional lattice. While the results in \citep{bl95} depend very much on the details of the particular graph considered, \citep{mor} derives general results that depend on local properties of the graphs.

More formally, agents are nodes of a fixed infinite graph and revise their strategies synchronously.
There is an edge
$(i,j)$ if agents $i$ and $j$ can interact with each other. Each node has a
choice between two possible behaviors labeled $A$ and $B$. On each
edge $(i,j)$, there is an incentive for $i$ and $j$ to have their
behaviors match, which is modeled as the following coordination game
parametrized by a real number $q\in (0,1)$:
if $i$ and $j$ choose $A$ (resp. $B$), they each receive a payoff of
$q$ (resp. $(1-q)$); if they choose opposite strategies, then they
receive a payoff of $0$.
Then the total payoff of a player is the sum of the payoffs with each
of her neighbors. If the degree of node $i$ is $d_i$ and $N_i^B$ is
her number of neighbors playing $B$, then the payoff to $i$ from
choosing $A$ is $q(d_i-N_i^B)$ while the payoff from choosing $B$ is
$(1-q)N^B_i$. Hence the best response strategy for agent $i$ is to adopt $B$ if $N_i^B>qd_i$ and $A$ if
$N_i^B\leq qd_i$. 
A number of qualitative insights can be derived from
this simple diffusion model \citep{bl95,mor}.
Clearly, a network where all nodes play $A$ is an equilibrium of the
game as is the state where all nodes play $B$. 
Consider now a network where all nodes initially play $A$ and a small number of nodes are forced to adopt strategy $B$ (constituting the seed).
Other nodes in the network apply best-response updates, then
these nodes will be repeatedly applying the following rule: switch to
$B$ if enough of your neighbors have already adopted $B$.
There can be
a cascading sequence of nodes switching to $B$ such that a
network-wide equilibrium is reached in the limit.
In particular, contagion is said to occur if action $B$ can spread from a finite set of players to the whole population. It is shown in \citep{mor} that there is a contagion threshold such that contagion occurs if and only if the parameter $q$ is less than the contagion threshold. In particular the contagion threshold is always at most $1/2$ so that $B$ needs to be the risk-dominant action in order to spread. However this condition is not sufficient and \citep{mor} gives a number of characterizations of the contagion threshold and also shows that low contagion threshold implies the existence of equilibria where both actions are played. 

Most of the results on this model are restricted to deterministic infinite graphs and the results in \citep{mor} depend on local properties of graphs that need to be checked node-by-node. In networks exhibiting heterogeneity among the nodes, this approach will be limited in its application.
In this work, we analyze the diffusion in the large population limit when the underlying graph is a random network $G(n,\bod)$ with $n$ vertices and where $\bod=(d_i)_1^n$ is a given degree sequence (see Section \ref{sec:conf} for a detailed definition).
Although random graphs are not considered to be highly realistic
models of most real-world networks, they are often used as first
approximation and are a natural first choice for a sparse interaction
network in the absence of any known geometry of the problem
\citep{netbook}. 
Our analysis yields several insights into how
the diffusion propagates and we believe that they could be used in real-world situations.
As an example, we show some implications of our results for the design of optimal advertising strategies.
Our results extend the previous analysis of global
cascades made in \citep{wat02} using a threshold model.
They differ greatly from the study of standard epidemics models used
for the analysis of the spread of viruses \citep{bailey} where an
entity begin as 'susceptible' and may become 'infected' and infectious
through contacts with her neighbors with a given probability.
Already in the simple coordination game presented above, we show that connectivity (i.e. the average number of
neighbors of a typical agent) plays an ambiguous role: while it
allows the diffusion to spread, when the network is highly connected,
the diffusion is also limited by high-degree nodes, which are very
stable. These nodes require a lot of their neighbors to switch to $B$
in order to play $B$ themselves. 

In the
case of a sparse random network of interacting agents, we compute the
contagion threshold in Section \ref{sec:cont} and relate it to the important notion of {\it pivotal players}. We also show the existence of (continuous and discontinuous)
phase transitions.
In Section \ref{sec:eq}, we analyze the possible equilibria of the game for low values of $q$. In particular, we give conditions under which an equilibrium with coexistence of large (i.e. containing a positive fraction of the total population) connected sets of players $A$ and $B$ is possible.
In Section \ref{sec:ad}, we also compute the minimal size of a seed of new adopters in order to trigger a global cascade in the case where these adopters can only be
sampled without any information on the graph. We show that this
minimal size has a non-trivial behavior as a function of the connectivity.
In Section \ref{sec:lmf}, we give a heuristic argument allowing to
recover the technical results which gives some intuition behind our formulas.
Our results also explain why social networks can display a
great stability in the presence of continual small shocks that are as
large as the shocks that ultimately generate a global
cascade. Cascades can therefore be regarded as a specific
manifestation of the robust yet fragile nature of many complex systems
\citep{cd02}: a system may appear stable for long periods of time and
withstand many external shocks (robustness), then suddenly and
apparently inexplicably exhibit a large scale cascade (fragility).

Our paper is divided into two parts. The first part is contained in the next section. We apply our technical results to the particular case of the coordination game presented above. 
We describe our main findings and provide heuristics and intuitions for them.
These results are direct
consequences of our main theorems stated and proved in the second part of the paper. This second part starts with Section \ref{sec:mod}, where we present in details our most general model of diffusion (where initial activation and thresholds can vary among the population and also be random). This general model allows to recover and extend results in the random graphs literature and is used in \citep{coulel} for the analysis of the impact of clustering.
We also state our main
technical results: Theorem \ref{th:epi} and Theorem
\ref{prop:cascade}. Their proofs can be found in Sections \ref{sec:thepi} and \ref{sec:casc} respectively. Section \ref{sec:conc} contains our concluding remarks.

{\bf Probability asymptotics:}
in this paper, we consider sequences of (random) graphs and asymptotics
as the number of vertices $n$ tends to infinity. For notational
simplicity we will usually not show the dependency on $n$ explicitly.
All unspecified limits and other
asymptotics statement are for $n\to \infty$. For example, w.h.p. (with
high probability) means with probability tending to 1 as $n\to\infty$
and $\to^p$ means convergence in probability as $n\to
\infty$. Similarly, we use $o_p$ , $\Omega_p$ and $O_p$ in a standard way, see \cite{rg}. For
example, if $X^{(n)}=X$ is a parameter of the random graph, $X=o_p(n)$ means
that $\Pb(X>\epsilon n)\to0$ as $n\to \infty$ for every $\epsilon>0$, equivalently $X/n\to^p 0$, or for every $\epsilon>0$, $|X|<\epsilon n$
w.h.p.

\section{Analysis of a simple model of cascades}

In this section, we first present the model for the underlying graph. This is the only part of this section required to understand our main technical results in Section \ref{sec:mod}.
The remaining parts of this section present and discuss applications and special cases exploring the significance of the technical results derived in the sequel.

\subsection{Graphs: the configuration model}\label{sec:conf}

We consider a set $[n]=\{1,\dots,n\}$ of agents interacting over a social network.
Let $\boldd=(d_i^{(n)})_1^n=(d_i)_1^n$ be a sequence of non-negative integers such that $\sum_{i=1}^n d_i$ is even. For notational simplicity we will usually not show the
dependency on $n$ explicitly. This sequence is the degree sequence of the graph: agent $i\in [n]$ has degree $d_i$, i.e. has $d_i$ neighbors.
We define a random
multigraph (allowing for self-loop and multiple links) with given degree sequence $\boldd$, denoted by
$G^*(n,\boldd)$ by the configuration model \citep{bol01}: take a set
of $d_i$ half-edges for each vertex $i$ and combine the half-edges into
pairs by a uniformly random matching of the set of all
half-edges. Conditioned on the multigraph $G^*(n,\boldd)$ being a
simple graph, we obtain a uniformly distributed random graph with the
given degree sequence, which we denote by $G(n, \boldd)$ \citep{j06}.

We will let $n\to \infty$ and assume that we are given $\boldd=(d_i)_1^n$
satisfying the following regularity conditions, see \citep{mr95}:
\begin{Condition}\label{cond}
For each $n$, $\boldd=(d_i^{(n)})_1^n$ is a sequence of
non-negative integers such that $\sum_{i=1}^n d_i$ is even and, for
some probability distribution $\boldp=(p_r)_{r=0}^\infty$ independent of $n$,
\begin{itemize}
\item[(i)] $|\{i:\: d_i=r\}|/n\to p_r$ for every $r\geq 0$ as $n\to
  \infty$;
\item[(ii)] $\lambda:=\sum_{r\geq 0}rp_r\in (0,\infty)$;
\item[(iii)] $\sum_{i\in [n]} d_i^2=O(n)$.
\end{itemize}
\end{Condition}
In words, we assume that the empirical distribution of the degree sequence converges to a fixed probability distribution $\boldp$ with a finite mean $\lambda$. Condition \ref{cond} (iii) ensures that $\boldp$ has a finite second moment and is technically required. It implies that the empirical mean of the degrees converges to $\lambda$.

In some cases, we will need the following additional condition which implies that the asymptotic distribution of the degrees $\boldp$ has a finite third moment:
\begin{Condition}\label{cond2}
For each $n$, $\boldd=(d_i^{(n)})_1^n$ is a sequence of
non-negative integers such that
$\sum_{i\in [n]} d_i^3=O(n)$.
\end{Condition}

The results of this work can be applied to some other random graphs
models too by conditioning on the vertex degrees (see \citep{jl07,jlgiant}). For example, for the Erd\H{o}s-R\'enyi graph $G(n,\lambda/n)$ with $\lambda\in (0,\infty)$, our results will hold with the limiting probability distribution $\boldp$ being a Poisson distribution with mean $\lambda$.

\subsection{Contagion threshold for random networks}\label{sec:cont}

An interesting perspective is to understand how different network
structures are more or less hospitable to cascades. Going back to
the diffusion model described in the Introduction, we see that the lower $q$ is, the easiest the
diffusion spreads.
In \citep{mor}, the contagion threshold of a connected infinite network
(called the cascade capacity  in \citep{easley2010}) is defined as the
maximum threshold $q_c$ at which a finite set of initial adopters can
cause a complete cascade, i.e. the resulting cascade of adoptions of
$B$ eventually causes every node to switch from $A$ to $B$.
In this section, we consider a model where the initial adopters are forced
to play $B$ forever. In this case, the diffusion is monotone and the
number of nodes playing $B$ is non-decreasing. We say that this case
corresponds to the {\it permanent adoption model}: a player playing $B$ will
never play $A$ again. We will discuss a model where the initial adopters also apply best-response update in Section \ref{sec:eq}.

To illustrate the notion of contagion threshold, consider the infinite network depicted on Figure \ref{fig:counter34}.
It is easy to see that the contagion threshold for such a network is
$q_c=\frac{1}{4}$, the same as for the $2$-dimensional lattice. 
Note that in the network of Figure \ref{fig:counter34} the density of nodes with degree
$3$ is $\frac{8}{9}$ and the density of nodes of degree $4$ is $\frac{1}{9}$, whereas in the $2$-dimensional lattice, all vertices have degree $4$. 
\begin{figure}[htb]
\begin{center}
\includegraphics[width=5cm]{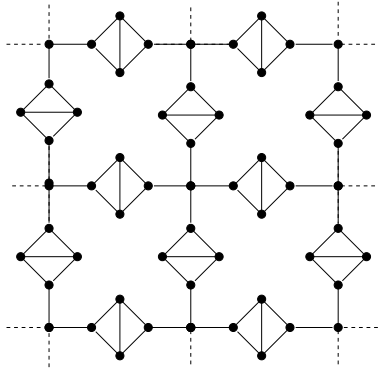}
\hspace{25pt} \caption{An infinite network with contagion threshold
  $\frac{1}{4}$ and with density of nodes with degree $3$ equals to $\frac{8}{9}$ and degree $4$ equals to $\frac{1}{9}$.}
\label{fig:counter34}
\end{center}\end{figure}

We now compute the contagion threshold for a sequence of random
networks. Since a random network is finite and not necessarily
connected, we first need to adapt the definition of contagion
threshold to our context. We start by recalling a basic result on the
size of the largest component of a random graph with a given degree
sequence \citep{mr95,jlgiant}, namely a sufficient condition for the existence of a {\it giant component}.

\begin{Proposition*}
Consider the random graph $G(n,\bod)$ satisfying Conditions \ref{cond}
with asymptotic degree distribution $\boldp=(p_r)_{r=0}^\infty$.
If $\sum_{r=1}^\infty r(r-2)p_r>0$ then the size of the largest component of
$G(n,\bod)$ is $\Omega_p(n)$. In this case, we say that there exists a {\it giant component}.
\end{Proposition*}

For a graph $G=(V,E)$ and a parameter $q$, we consider the largest connected
component of the induced subgraph in which we keep only vertices of degree
strictly less than $q^{-1}$.  We call the vertices in this component {\it pivotal
players}: if only one pivotal player switches from $A$ to $B$ then the
whole set of pivotal players will eventually switch to $B$ in the
permanent adoption model.
For a player $v\in V$, we denote by $C(v,q)$ the final
number of players $B$ in the permanent adoption model with parameter
$q$, when the initial state consists of only $v$ playing $B$, all
other players playing $A$. Informally, we say that $C(v,q)$ is the
size of the cascade induced by player $v$.

\begin{Proposition}\label{prop:cap}
Consider the random graph $G(n,\bod)$ satisfying Conditions \ref{cond}
and \ref{cond2} 
with asymptotic degree distribution $\boldp=(p_r)_{r=0}^\infty$ such
that $\sum_{r\geq 1}r(r-2)p_r>0$. We define $q_c$ by:
\begin{eqnarray}
\label{eq:cap}q_c(\boldp)=q_c &=& \sup\left\{q:\: \sum_{2\leq r<
    q^{-1}}r(r-1)p_r>\sum_{1\leq r} rp_r\right\}.
\end{eqnarray}
Let $P^{(n)}$ be the set of pivotal players in $G(n,\bod)$.
\begin{itemize}
\item[(i)] For $q<q_c$, there are constants $0<\gamma(q,\boldp)\leq s(q,\boldp)$ such that w.h.p.
  $\lim_n\frac{|P^{(n)}|}{n}= \gamma(q,\boldp)$ and for any $v\in
  P^{(n)}$, $\liminf_n \frac{C(v,q)}{n}\geq s(q,\boldp)$.
\item[(ii)] For $q>q_c$, for an uniformly chosen player $v$, we have $C(v,q)=o_p(n)$. The same result holds if $o(n)$ players are chosen uniformly at random.
\end{itemize}
\end{Proposition}

Note that Proposition \ref{prop:cap} is valid only for graphs with a
giant component (for graphs with all components $o_p(n)$, we have
trivially that $C(v,q)=o_p(n)$ for all values of $q$).
We can rewrite (\ref{eq:cap}) as follows: let $D$ be a random variable with distribution $\boldp$, i.e. $\Pb(D=r)=p_r$, then
\begin{eqnarray*}
q_c = \sup\left\{q:\: \Eb\left[ D(D-1)\ind(D<q^{-1})\right]>\Eb[D]\right\}.
\end{eqnarray*}

We can restate Proposition \ref{prop:cap} as follows: let $C^{(n)}$ be the size of a cascade induced by
switching a random player from $A$ to $B$. Proposition \ref{prop:cap}
implies that for $q>q_c$, we have $\Pb(C^{(n)}>\epsilon n)\to 0$ as
$n\to \infty$ for every $\epsilon>0$, whereas for $q<q_c$ there exists
constants $s(q,\boldp)\geq \gamma(q,\boldp)>0$ depending only on $q$ and the
parameters of the graph such that for every $\epsilon>0$, $\liminf_n\Pb(C^{(n)}\geq (s(q,\boldp)-\epsilon) n)\geq \gamma(q,\boldp)$.
Informally, we will say that global cascades
(i.e. reaching a positive fraction of the population) occur
when $q<q_c$ and do not occur when $q>q_c$.
We call $q_c$ defined by (\ref{eq:cap}) the contagion threshold for
the sequence of random networks with degree distribution
$(p_r)$.
This result is in accordance with the heuristic result of \citep{wat02}
(see in particular the cascade condition Eq. 5 in \citep{wat02}).
Proposition \ref{prop:cap} follows
from Theorem \ref{prop:cascade} below which also gives estimates for
the parameters $s(q,\boldp)$ and $\gamma(q,\boldp)$ when $q<q_c$.
Under additional technical conditions, Theorem \ref{prop:cascade} shows that $\lim_n\frac{C(v,q)}{n}= s(q,\boldp)$ in case (i).

Note that the condition $\sum_{r\geq 1} r(r-2)p_r>0$ implies that
$p_0+p_2<1$ so that for $q>1/3$, we have, $\sum_{r<q^{-1}}r(r-1)p_r\leq 2p_2<
\sum_r rp_r$. Hence we have the following elementary corollary:
\begin{Corollary}\label{cor1/3}
For any random graph $G(n,\bod)$ satisfying Conditions \ref{cond} and
\ref{cond2} with asymptotic degree distribution $\boldp=(p_r)_{r=0}^\infty$ such
that $\sum_{r\geq 1}r(r-2)p_r>0$, we have $q_c\leq 1/3$.
\end{Corollary}
Recall that for general graphs, it is shown in \citep{mor} that
$q_c\leq 1/2$.
Note also that our bound in Corollary \ref{cor1/3} is tight: for a
$r$-regular network with $r\geq 3$ chosen uniformly at random (i.e. $G(n,\bod)$ such
that $d_i=r$ for all $i\in [n]$) , Proposition \ref{prop:cap} implies that
$q_c=r^{-1}$ corresponding to the contagion threshold of a
$r$-regular tree, see \citep{mor}.
For $r\geq 3$, a $r$-regular random graph is connected w.h.p, so that
for any $q<r^{-1}$, an initial adopter will cause a complete cascade (i.e. reaching all
players).
In particular, in this case, the only possible
equilibria of the game are all players playing $A$ or all players
playing $B$. We will come back to the analysis of equilibria of the game in Section \ref{sec:eq}.

Given a degree distribution $\boldp$, Proposition \ref{prop:cap} gives the typical value of $q_c(\boldp)$ under the uniform distribution on graphs with the prescribed degree distribution. 
A question of interest is to find among all graphs with a given degree sequence, which ones achieve the lowest or highest value of the contagion threshold.
Indeed, the contagion threshold of a graph can be seen as a proxy for its susceptibility to cascades: the higher the contagion threshold is, the more susceptible the graph is to cascading behavior.
Clearly if $d_{\max}$ is the maximal degree in the graph, we have $q_c>\frac{1}{d_{\max}}$. In the case of regular graphs, we saw that the typical contagion threshold is equal to this lower bound: $q_c(d)=\frac{1}{d}$. Note however that there are examples of regular graphs in \citep{mor} with a strictly higher threshold (see Figure 3 there: a $8$-regular graph with contagion threshold $3/8$). 
Hence among all $r$-regular graphs, picking a graph at random will give w.h.p. a graph which is the least susceptible to cascades among those $r$-regular graphs. However it is not always the case that picking a graph at random among all graphs with a given degree sequence will result in the least susceptible one.
There are situations where the typical value of the contagion threshold exceeds the lowest possible value: this is the case in the example presented above in Figure \ref{fig:counter34} for which the contagion threshold is $\frac{1}{4}$ whereas the sequence of random graphs
$G(n,\bod)$ with asymptotic degree distribution such that
$p_3=\frac{8}{9}$ and $p_4=\frac{1}{9}$ has a contagion threshold of
$\frac{1}{3}$. The recent work \citep{bekkt} has a similar focus.

\begin{figure}[htb]
\begin{center}
\includegraphics[width=5cm]{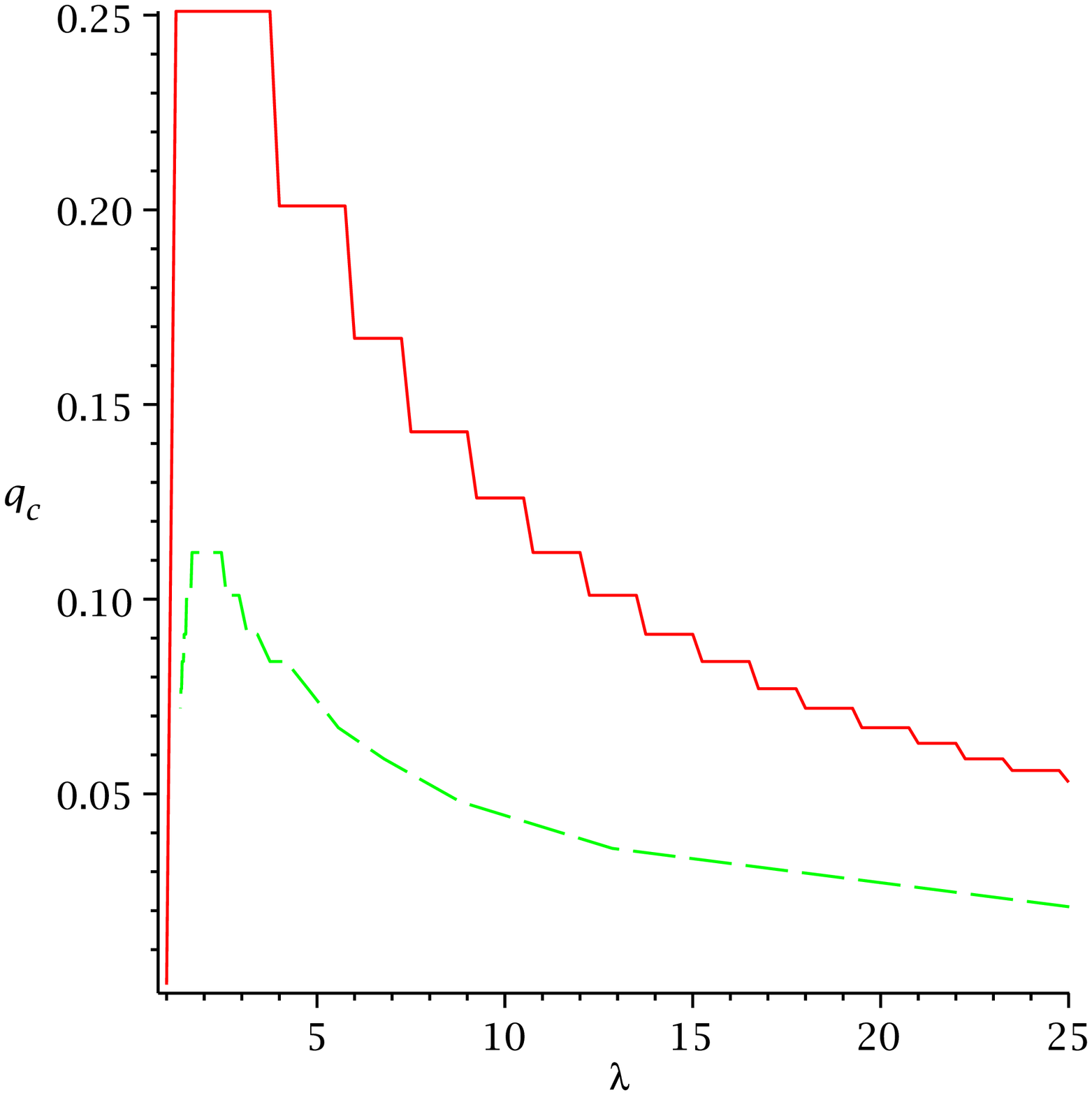}
\hspace{2cm}\includegraphics[width=4cm]{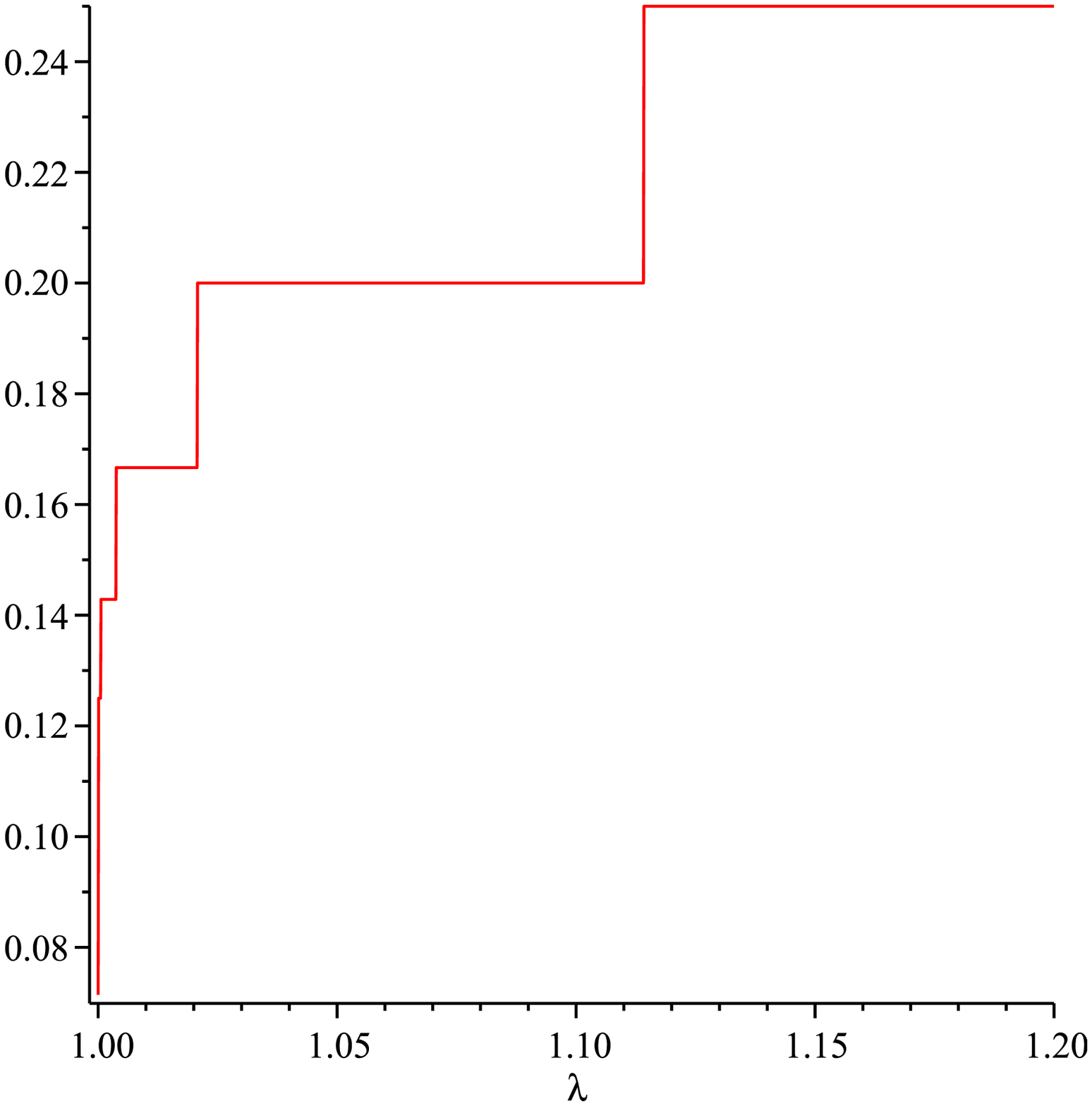}
\psfrag{qc}{$q_c$}
\hspace{25pt} \caption{Left: $q_c(\lambda)$ for Erd\H{o}s-R\'enyi random graphs and
  for power law graphs (dashed curve) as a function of the
  average degree $\lambda$.
Right: $q_c(\lambda)$ for Erd\H{o}s-R\'enyi random graphs as a
  function of $\lambda$ in the range $\lambda\in [1;1.2]$.}
\label{fig:qc}
\end{center}\end{figure}
We now consider some examples.
First the case of Erd\H{o}s-R\'enyi random graphs
$G(n,\lambda/n)$ where each of the ${n\choose 2}$ edges is present
with probability $\lambda/n$ for a fixed parameter $\lambda$.
In this case, we can apply our results with
$p_r=e^{-\lambda}\frac{\lambda^r}{r!}$ for all $r\geq 0$, i.e. the Poisson distribution with mean $\lambda$.
As shown in Figure
\ref{fig:qc} (Left) for the case of Erd\H{o}s-R\'enyi random graphs
$G(n,\lambda/n)$, $q_c$ is
well below $1/3$, indeed we have $q_c\leq 1/4$ for any value of
$\lambda$.
As shown in Figure \ref{fig:qc} (Right), we see that $q_c$ is a
non-decreasing function of $\lambda$ the average degree in the graph for $\lambda\leq 2$.
Clearly on Figure \ref{fig:qc} (Left), $q_c$ is a non-increasing function of
$\lambda$, for $\lambda\geq 4$.
The second curve in Figure \ref{fig:qc} corresponds to the contagion
threshold for a scale-free
random network whose degree distribution $p_r=
\frac{r^{-\gamma}}{\zeta(\gamma)}$ (with $\zeta(\gamma)=\sum
r^{-\gamma}$) is parametrized by the decay parameter $\gamma>1$. We see
that in this case we have $q_c\leq 1/9$.
In other words, in an Erd\H{o}s-R\'enyi random graph, in order to
have a global cascade, the parameter $q$ must be such that any node
with no more than four neighbors must be able to adopt $B$ even if it
has a single adopting neighbor. In the case of the scale free random
network considered, the parameter $q$ must be much lower and any node
with no more than nine neighbors must be able to adopt $B$ with a
single adopting neighbor.
This simply reflects the intuitive idea that for widespread diffusion
to occur there must be a sufficient high frequency of nodes that are
certain to propagate the adoption.

We also observe that in both cases, for $q$ sufficiently low, there are two
critical values for the parameter $\lambda$, $1<\lambda_{i}(q)
<\lambda_{s}(q)$ such that a global cascade for a fixed $q$ is only
possible for $\lambda\in (\lambda_i(q);\lambda_s(q))$.
The heuristic reason for these two thresholds is that a cascade can be
prematurely stopped at high-degree nodes.
For Erd\H{o}s-R\'enyi random graphs,  when $1\leq
\lambda<\lambda_i(q)$, there exists a giant component, i.e. a connected component containing a positive fraction of the nodes.
The high-degree nodes are quite infrequent so that the diffusion should
spread easily. However, for $\lambda$ close to one, the diffusion does
not branch much and progresses along a very thin tree, ``almost a line'', so
that its progression is stopped as soon as it encounters a high-degree
node. Due to the variability of the Poisson distribution, this happens
before the diffusion becomes too big for
$\lambda<\lambda_i(q)$. Nevertheless the condition
$\lambda>\lambda_i(q)$ is not sufficient for a global cascade. Global
diffusion also requires that the network not be too highly connected.
This is reflected by the existence of the second threshold
$\lambda_s(q)$ where a further transition occurs, now in the opposite
direction. For $\lambda>\lambda_s(q)$, the diffusion will not reach a
positive fraction of the population. The intuition here is clear: the
frequency of high-degree nodes is so large that diffusion cannot avoid
them and typically stops there since it is unlikely that a high enough
fraction of their many neighbors eventually adopts. Following \citep{wat02},
we say that these nodes are locally stable.

\begin{figure}[htb]
\begin{center}
\includegraphics[width=4.5cm]{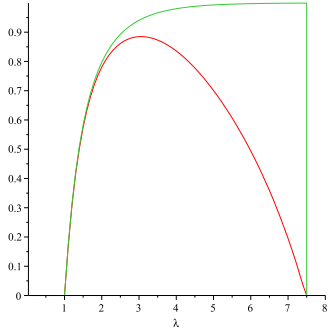}
\psfrag{qc}{$q_c$}
\hspace{25pt} \caption{Size $s(q,\lambda)$ of the cascade (in percent
  of the total population) for Erd\H{o}s-R\'enyi
  random graphs as a function of $\lambda$ the average degree for a
  fixed $q=0.15$. The lower curve gives the asymptotic fraction of
  pivotal players $\gamma(q,\lambda)$.}
\label{fig:cascade}
\end{center}\end{figure}

The proof of our Theorem \ref{prop:cascade} makes previous intuition rigorous for a
more general model of diffusion and gives also more insights on the nature of the phase transitions. We describe it now.
The lower curve in Figure \ref{fig:cascade}
represents the number of pivotal players in an Erd\H{o}s-R\'enyi
random graph as a function of $\lambda$ the average connectivity for
$q^{-1}=6.666...$: hence we keep only the largest connected component
of an Erd\H{o}s-R\'enyi random graph where we removed all vertices of
degree greater than $6$. By the same heuristic argument as above, we
expect two phase transitions for the size of the set of pivotal
players. In the proof of Theorem \ref{prop:cascade}, we show that it
is indeed the case and moreover that the phase transitions
occur at the same values $\lambda_i(q)$ and $\lambda_s(q)$ as can be
seen on Figure \ref{fig:cascade} where the normalized size (i.e. fraction)
$\gamma(q,\lambda)$ of the set of pivotal players is positive only for $\lambda\in (\lambda_i(q),\lambda_s(q))$.
Hence a cascade is possible if and only
if there is a 'giant' component of pivotal players.
Note also that both phase transitions for the pivotal players are
continuous, in the sense that the function $\lambda \mapsto
\gamma(q,\lambda)$ is continuous.
This is not the case for the second phase transition for the normalized size of
the cascade given by $s(q,\lambda)$ in Proposition \ref{prop:cap}: the function $\lambda\mapsto s(q,\lambda)$ is continuous
in $\lambda_i(q)$ but not in $\lambda_s(q)$ as depicted on Figure
\ref{fig:cascade}. This has important consequences: around
$\lambda_i(q)$ the propagation of cascades is limited by the
connectivity of the network as in standard epidemic models. But around
$\lambda_s(q)$, the propagation of cascades is not limited by the
connectivity but by the high-degree nodes which are locally
stable.

To better understand this second phase transition, consider a
dynamical model where at each step a player is chosen at random and
switches to $B$. Once the corresponding cascade is triggered, we come
back to the initial state where every node play $A$ before going to the next step.
Then
for $\lambda$ less than but very close to $\lambda_s(q)$, most
cascades die out before spreading very far. However, a set of pivotal players still exists, so very rarely a cascade will be triggered by
such a player in which case the high connectivity of the network
ensures that it will be extremely large.
As $\lambda$ approaches $\lambda_s(q)$ from below, global cascades
become larger but increasingly rare until they disappear implying a
discontinuous phase transition. 
For low values of $\lambda$, the cascades are infrequent and small. As $\lambda$ increases, their frequencies and sizes also increase until a point where the cascade reaches almost all vertices of the giant component of the graph. Then as $\lambda$ increases, their sizes remain almost constant but their frequencies are decreasing.

\subsection{Equilibria of the game and coexistence}\label{sec:eq}

We considered so far the permanent adoption
model: the only possible transitions are from playing $A$ to $B$.
There is another possible model in which the initial adopters playing
$B$ also apply best-response update. We call this model the
{\it non-monotonic} model (by opposition to the permanent adoption model).
In this model, if the
dynamic converges, the final state will be an equilibrium of the game.
An equilibrium of the game is a fixed point of the best response
dynamics. For example, the states in which all players play $A$ or all
players play $B$ are trivial equilibria of the game.
Note that the permanent adoption model does not necessarily
yield to an equilibrium of the game as the initial seed does not apply
best response dynamics.
To illustrate the differences between the two models consider the graph
of Figure \ref{fig:contrex} for a value of $q=1/3$: if the circled
player switches to $B$, the whole network will eventually switch to
$B$ in the permanent adoption model, whereas the dynamic for the
non-monotonic model will oscillate between the state where only the
circled player plays $B$ and the state where only his two neighbors of
degree two play $B$ (agents still revise their strategies synchronously).
\begin{figure}[htb]
\begin{center}
\includegraphics[width=2cm]{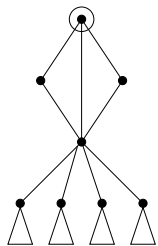}
\hspace{25pt} \caption{The circled player induces a global cascade in the
  permanent adoption model but not in the non-monotonic model for
  $q=1/3$ (each triangle represents a long line of players).}
\label{fig:contrex}
\end{center}\end{figure}

Clearly if a player induces a global cascade for the non-monotonic
model, it will also induce a global cascade in the permanent adoption
model.
As illustrated by the case of Figure \ref{fig:contrex}, the
converse is not true in general.
Hence, a priori, a pivotal player as defined in previous section might
not induce a global cascade in the non-monotonic model.
In \citep{mor}, it is shown that if one can find a finite set of
initial adopters causing a complete cascade for the permanent adoption
model, it is also possible to find another, possibly larger but still
finite, set of initial adopters leading to a complete cascade for the
non-monotonic model. Hence the contagion threshold as defined in
\citep{mor} for infinite graphs is the same for both models.
In our case, we see that if we switch from $A$ to $B$ two pivotal players who are neighbors, then the whole set of pivotal players will
eventually switch to $B$ in the non-monotonic model.
Hence in our analysis on random networks, we also say that the contagion threshold is the same in both
models.
Moreover, both models will have exactly the same dynamics if started
with the set of pivotal players playing $B$ and all other players
playing $A$. In particular, it shows that the dynamic converges and
reaches an equilibrium of the game. Hence we have the following corollary:
\begin{Corollary}\label{prop:equil}
Consider the random graph $G(n,\bod)$ satisfying Conditions \ref{cond}
and \ref{cond2}
with asymptotic degree distribution $\boldp=(p_r)_{r=0}^\infty$.
For $q<q_c$, there exists w.h.p. an equilibrium
of the game in which the number of players $B$ is more than $s(q,\boldp) n$
(defined in Proposition \ref{prop:cap}) and it can be obtained from
the trivial equilibrium with all players playing $A$ by switching only
two neighboring pivotal players. 
\end{Corollary}

Hence for $q<q_c$, the trivial equilibrium all $A$, is rather 'weak' since
two pivotal players can induce a global cascade and there are
$\Omega_p(n)$ such players so that switching two neighbors at random
will after a finite number of trials (i.e. not increasing with $n$) induce
such a global cascade. We call this equilibrium the pivotal equilibrium. Figure \ref{fig:mgeo} shows the average number of trials required. It goes to infinity at both extremes $\lambda_i(q)$ and $\lambda_s(q)$. Moreover, we see that for most values of $\lambda$ inside this interval, the average number of trials is less than $2$.
If $q>q_c$, then by definition if there are pivotal players, their
number must be $o_p(n)$.
Indeed, in the case of $r$-regular graphs, there are no pivotal
players for $q>q_c=r^{-1}$.
Hence, it is either impossible or very difficult (by sampling) to find a set of players with cardinality bounded in $n$ leading to a global cascade since in all cases, their number is $o_p(n)$.

\begin{figure}[htb]
\begin{center}
\includegraphics[width=5cm]{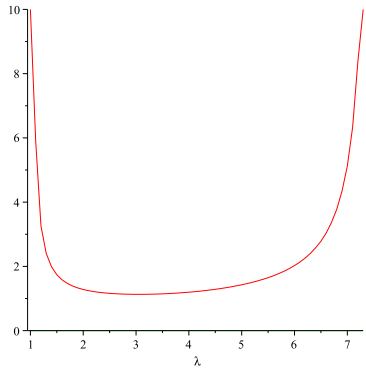}
\hspace{25pt} \caption{Average number of trials required to trigger a global cascade and reach the pivotal equilibrium for Erd\H{o}s-R\'enyi random graphs $G(n,\lambda/n)$ as a function of $\lambda$.}
\label{fig:mgeo}
\end{center}\end{figure}

We now show that equilibria with {\it co-existent} strategies exist. Here by co-existent, we mean that a connected fraction of players $A$ and $B$ are present in the equilibrium. 
More formally, we call a {\it giant component}: a subset of vertices containing a positive fraction (as $n$ tends to infinity) of the total size of the graph such that the corresponding induced graph is connected. We ask if there exists equilibrium with a giant component of players $A$ and $B$?
In the case $q<q_c$, for values of $\lambda$ close to $\lambda_s(q)$, the global cascade reaches almost all nodes of the giant component leaving only very small of connected component of players $A$. However for $\lambda$ close to $\lambda_i(q)$, this is not the case and co-existence is possible as shown by the following proposition:

\begin{Proposition}\label{prop:inactiveer}
In an Erd\H{o}s-R\'enyi random graph $G(n,\lambda/n)$, for $q<q_c$, there exists $\lambda_c(q)\in [\lambda_i(q),\lambda_s(q)]$ such that:
\begin{itemize}
\item for $\lambda \in (\lambda_i(q),\lambda_c(q))$, in the pivotal equilibrium, there is coexistence of a giant component of players $A$ and a giant component of players $B$.
\item for $\lambda>\lambda_c(q)$, in the pivotal equilibrium, there is no giant component of players $A$, although there might be a positive fraction of players $A$.
\end{itemize}
\end{Proposition}

\begin{figure}[htb]
\begin{center}
\includegraphics[width=5cm]{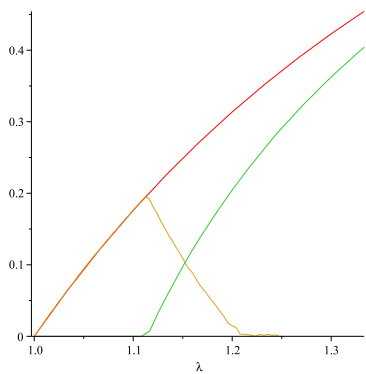}
\hspace{25pt} \caption{Coexistence: the upper (red) curve is the fraction of players in the giant component as a function of the average degree $\lambda$. The lower (green) curve is the fraction of players $B$ in the pivotal equilibrium as a function of $\lambda$ (for $q=0.2$). The last (brown) curve is the size of the giant component of players $A$ in the pivotal equilibrium as a function of $\lambda$.}
\label{fig:inactive}
\end{center}\end{figure}

Figure \ref{fig:inactive} illustrates this proposition in the case of Erd\H{o}s-R\'enyi random graphs $G(n,\lambda/n)$. The difference between the upper (red) and lower (green) curve is exactly the fractions of players $A$ in the pivotal equilibrium while the brown curve represent the size (in percentage of the total population) of the largest connected component of players $A$ in the pivotal equilibrium. This curve reaches zero exactly at $\lambda_c(q)$. Hence for $\lambda>\lambda_c(q)$, we see that there is still a positive fraction of the population playing $A$, but the set of players $A$ is divided in small (i.e. of size $o(n)$) connected components, like 'islands' of players $A$ in a 'sea' of players $B$. Note also that for Erd\H{o}s-R\'enyi random graphs, the value of $\lambda_c(q)$ is close to $\lambda_i(q)$.
Proposition \ref{prop:inactiveer} follows from the following proposition whose proof is given in Section \ref{sec:inactive}.

We first introduce some notations: for $D$ a random variable, we define for $0\leq x\leq 1$, $D_x$ the thinning of $D$ obtained by taking $D$ points and then randomly and independently keeping each of them with probability $x$, i.e. given $D=d$, $D_x\sim \Bi(d,x)$.
We define $g(z,\boldp)=\sum_{s}p_s \sum_{r\geq s-\lfloor sq\rfloor}rb_{sr}(z)$ (corresponding to the function $g$ defined in previous section for $\balpha=0$ and $\pi=1$), $\lambda(\boldp)=\sum_s sp_s$ and
\begin{eqnarray*}
\xi=\sup\left\{z<1,\: \lambda(\boldp) z^2 = g(z,\boldp) \right\}.
\end{eqnarray*}

\begin{Proposition}\label{prop:inactive}
Consider the random graph $G(n,\bod)$ satisfying Conditions \ref{cond}
and \ref{cond2}
with asymptotic degree distribution $\boldp=(p_r)_{r=0}^\infty$. Let $D$ be a random variable with distribution $\boldp$.
If $q<q_c$, then $\xi<1$. Assume moreover that there exists $\epsilon>0$ such that $g(z,\boldp)<0$ for $z\in (\xi-\epsilon,\xi)$.
There is coexistence of a giant component of players $A$ and a giant component of players $B$ in the pivotal equilibrium if
\begin{eqnarray*}
\Eb\left[D_\xi(D_\xi-2)\ind(qD\geq D-D_\xi)\right]> 0.
\end{eqnarray*}
\end{Proposition}

\subsection{Advertising with word of mouth communication}\label{sec:ad}

We consider now scenarios where $\lambda\notin
[\lambda_i(q),\lambda_s(q)]$ and the initial set of adopters grows
linearly with the total population $n$.
More precisely, consider now a firm advertising
to a group of consumers, who share product information among
themselves: potential buyers are not aware of the existence of the
product and the firm undertakes costly informative advertising.
The firm chooses the fraction of individuals who receive
advertisements.
Individuals are located in a social network modeled by a random
network $G(n,\bod)$ with given vertex degrees as in previous section.
However contrary to most work on viral
marketing \citep{rd02}, \citep{kkt03}, we assume that the advertiser has
limited knowledge about the network: the firm only knows the
proportions of individuals having different degrees in the social
network.
One possibility for the firm is to sample individuals randomly
and to decide the costly advertising for this individual based on her
degree (i.e. her number of neighbors).
The action of the firm is then encoded in a vector
$\balpha=(\alpha_d)$, where $\alpha_d$ represents the fraction of
individuals with degree $d$ which are directly targeted by the firm.
These individuals will constitute the seed (in a permanent adoption model) and we call them the early adopters.
Note that the case $\alpha_d=\alpha$ for all $d$ corresponds to a case
where the firm samples individuals uniformly. This might be one
possibility if it is unable to observe their degrees.
In order to optimize its strategy, the firm needs to compute the
expected payoff of its marketing strategy as a function of
$\balpha$. Our results allows to estimate this function in terms of
$\balpha$ and the degree distribution in the social network.

We assume that a buyer might buy either if she receives advertisement from
the firm or if she receives information via word of mouth
communication \citep{ef95}.
More precisely, following \citep{gg08}, we consider the following
general model for the diffusion of information: a buyer obtains
information as soon as one of her neighbors buys the product but she
decides to buy the product when $\Bi(k,\pi)>K(d)$ where $d$ is
her number of neighbors (i.e. degree) and $k$ the number of neighbors having the
product, $\Bi(k,\pi)$ is a Binomial random variable with parameters
$k$ and $\pi\in [0,1]$ and $K(d)$ is a general random variable (depending on $d$).
In words, $\pi$ is the probability that a particular neighbor does
influence the possible buyer. This possible buyer does actually buy
when the number of influential neighbors having bought the product
exceeds a threshold $K(d)$.
Thus, the thresholds $K(d)$ represent the different propensity of
nodes to buy the new product when their neighbors do.
The fact that these are possibly randomly selected is intended to model our
lack of knowledge of their values and a possibly heterogeneous population.
Note that for $K(d)=0$ and $\pi\in [0,1]$, our model of diffusion is
exactly a contact process with probability of contagion between
neighbors equals to $\pi$. This model is also called the SI
(susceptible-infected) model in mathematical epidemiology
\citep{bailey}.
Also for $\pi=1$ and $K(d)=qd$, we recover
the model of \citep{mor} described previously where players $A$ are non-buyers and players $B$ are buyers.

We now give the asymptotic for the final number of buyers for the case
$\pi\in [0,1]$ and $K(s)=qs$, with $q\leq 1$ (the general case with a
random threshold is given in Theorem \ref{th:epi}).
We first need to introduce some notations.
For integers $s\geq 0$ and $0\leq r\leq s$ let $b_{s r}$ denote the
binomial probabilities
$b_{s r}(p) := \Pb(\Bi(s,p)=r) = {s \choose r} p^r(1-p)^{s-r}$.
Given a distribution $\boldp=(p_s)_{s\in \Nbold}$, we
define the functions:
\begin{eqnarray*}
\label{eq:h}h(z;\balpha, \boldp, \pi) &:=&
\sum_{s}(1-\alpha_s)p_s \sum_{r\geq s-\lfloor sq\rfloor}rb_{sr}(1-\pi+\pi z),\\
g(z;\balpha,\boldp,\pi) &:=& \lambda(\boldp)z(1-\pi+\pi z)-h(z;\balpha,\boldp,\pi),\\
\label{eq:h1}h_1(z;\balpha,\boldp,\pi) &:=& \sum_s(1-\alpha_s)p_s\sum_{r\geq
  s-\lfloor sq\rfloor} b_{sr}(1-\pi+\pi z),
\end{eqnarray*}
where $\lambda(\boldp)=\sum_{s}s p_{s}$.
We define
\begin{eqnarray*}
\hz(\balpha,\boldp,\pi) :=\max\left\{z\in [0,1]:\:
g(z;\balpha,\boldp,\pi)= 0\right\}.
\end{eqnarray*}
We refer to Section \ref{sec:lmf} for an intuition behind the
definitions of the functions $h,h_1,g$ and $\hz$ in terms of a branching
process approximation of the local structure of the graph.

\begin{Proposition}\label{prop:ad}
Consider the random graph $G(n,\boldd)$ for a sequence
$(d_i)_1^n$ satisfying Condition \ref{cond}. If the strategy of
the firm is given by $\balpha$, then the final number of buyers is
given by $(1-h_1(\hz,\balpha,\boldp,\pi))n+o_p(n)$ provided
$\hz(\balpha,\boldp,\pi)=0$, or $\hz(\balpha,\boldp,\pi)\in (0,1]$,
and further $g(z;\balpha,\boldp,\pi)<0$ for any $z$ in some interval
$(\hz-\epsilon,\hz)$.
\end{Proposition}

To illustrate this result, we consider the simple case of
Erd\H{o}s-R\'enyi random graphs $G(n,\lambda/n)$ with $\pi=1$,
$\alpha_d=\alpha$ for all $d$, and $q>q_c(\lambda)$ where $q_c(\lambda)$ is the
contagion threshold for this network (defined in previous section). 
For simplicity, we omit to write explicitly the dependence in $\pi$ and replace the dependence in $\boldp$ by the parameter $\lambda$.

As we will see, for some values of $\lambda$ there is a phase transition in $\alpha$. For a certain value $\alpha_c(\lambda)$, we have: if $\alpha<\alpha_c(\lambda)$, the size of the diffusion is rather small whereas
it becomes very large for $\alpha>\alpha_c(\lambda)$. 
Clearly the advertising
firm has a high interest in reaching the 'critical mass' of
$\alpha_c(\lambda) n$ early adopters. This phase transition is reminiscent of
the one described in previous section. Recall that we are now in a
setting where $q>q_c(\lambda)$ so that a global cascade triggered by a single
(or a small number of) individual(s) uniformly sampled is not possible.
Hence, the firm will have to advertise to a positive fraction of individuals constituting the
seed for the diffusion. The intuition is
the following: if the seed is too small, then each early adopter
starts a small cascade which is stopped by high-degree nodes. When the
seed reaches the critical mass, then the cascades 'coalesce' and are
able to overcome 'barriers' constituted by high-degree nodes so that a
large fraction of the nodes in the 'giant component' of the graph
adopt.
Then increasing the size
of the seed has very little effect since most of the time, the new
early adopters (chosen at random by the firm) are already reached by
the global diffusion. Our Theorem \ref{th:epi} makes this heuristic
rigorous for the general model of diffusion described
above (with random thresholds).

We now give some more insights by exploring different
scenarios for different values of $q$.
\begin{figure}[htb]
\begin{center}
\includegraphics[width=5cm]{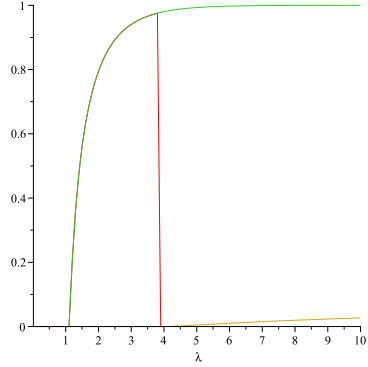}
\hspace{25pt}
\caption{Case $q<1/4$. The green (upper) curve is the size of the final set of
  players $B$ for $\alpha=\alpha_c(\lambda)$ given by
  $\lambda\mapsto 1-h_1(\hz,\alpha_c(\lambda))$ as a function of
  $\lambda$. The brown (lower) curve corresponds to the minimal size
  of the seed for a global cascade
  $\lambda\mapsto\alpha_c(\lambda)$. The red curve corresponds to the
  size of a global cascade when $\alpha=0$ similar as in Figure
  \ref{fig:cascade}.}
\label{fig:qpetit}
\end{center}\end{figure}
Consider first a case where $q<1/4$, then thanks to the results of
previous section, we know that there exists
$1<\lambda_i(q)<\lambda_s(q)$ such that for any $\lambda\in
(\lambda_i(q),\lambda_s(q))$, a global cascade is possible with
positive probability if only one random player switches to $B$.
In particular, if a fraction $\alpha$ of individuals uniformly sampled are playing $A$, then for any
$\alpha>0$, such a global cascade will occur with high
probability. For $\lambda>\lambda_s(q)$, we know that a single player
cannot trigger a global cascade.
More precisely, if players are chosen at random without any information on the
underlying network, any set of
initial adopters with size $o(n)$ cannot trigger a global cascade, as
shown by Theorem \ref{prop:cascade}.
However the final size of the set of players playing $B$ is a
discontinuous function of $\alpha$ the size of the initial seed.
If $\alpha_c(\lambda)$ is defined as the point at which this function
is discontinuous, we have: for $\alpha<\alpha_c(\lambda)$, the final
set of players $B$ will be only slightly larger than $\alpha$ but if
$\alpha>\alpha_c(\lambda)$ the final set of players $B$ will be very large.
Hence we will say that there is a global cascade when
$\alpha>\alpha_c(\lambda)$ and that there is no global cascade when
$\alpha<\alpha_c(\lambda)$. As shown in Figure \ref{fig:qpetit}, for $q<1/4$,
we have $\alpha_c(\lambda)=0$ for $\lambda\in
[\lambda_i(q),\lambda_s(q)]$ and $\alpha_c(\lambda)>0$ for
$\lambda>\lambda_s(q)$.
In Figure \ref{fig:qpetit}, we also see that our definition of global
cascade when $\lambda>\lambda_s(q)$ is consistent with our previous
definition since the function $\lambda\mapsto
1-h_1(\hz,\alpha_c(\lambda))$ giving the size of the final set of
players $B$ when the seed has normalized size $\alpha_c(\lambda)$ is
continuous in $\lambda$ and agrees with the previous curve for the size of the final set of players $B$ when $\lambda<\lambda_s(q)$ and with only one early adopter.

\begin{figure}[htb]
\begin{center}
\includegraphics[width=5cm]{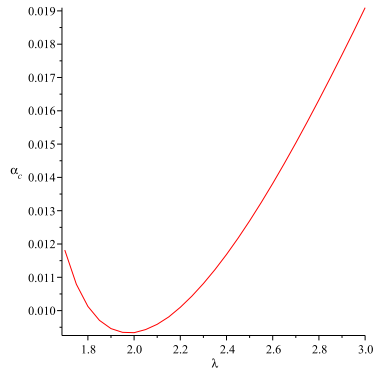}
\hspace{25pt}
\caption{Case $q>1/4$. Minimal size of the seed for a global cascade
  $\lambda\mapsto\alpha_c(\lambda)$ as a function of $\lambda$.}
\label{fig:alphac}\end{center}\end{figure}

We now consider the case where $q>1/4$.
In this case, thanks to the results of previous section, we know that
for any value of $\lambda$, a single initial player $B$ (sampled uniformly) cannot trigger a global cascade. But our definition of $\alpha_c(\lambda)$ still
makes sense and we now have $\alpha_c(\lambda)>0$ for all
$\lambda$. Figure \ref{fig:alphac} gives a plot of the function
$\lambda\mapsto \alpha_c(\lambda)$. We see that again there are two
regimes associated with the low/high connectivity of the graph.
For low values of $\lambda$, the function $\lambda\mapsto
\alpha_c(\lambda)$ is non-increasing in $\lambda$. This situation
corresponds to the intuition that is correct for standard epidemic
models according to which an increase in the connectivity makes the
diffusion easier and hence the size of the critical initial seed will
decrease.

\begin{figure}[htb]
\begin{center}
\includegraphics[width=5cm]{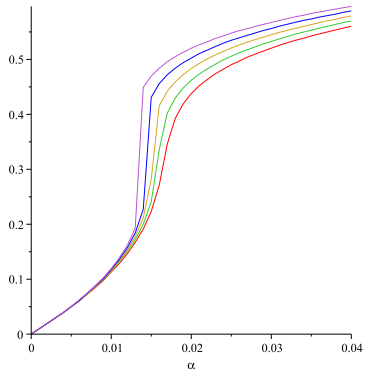}
\hspace{25pt}
\caption{Case $q>1/4$, low connectivity regime. Size of the final set of
  players $B$, $\alpha\mapsto 1-h_1(\hz,\alpha)$ as a function of the
  size of the initial seed $\alpha$ for different values of $\lambda=1.60, 1.61, 1.62, 1.63, 1.64$.}
\label{fig:va1}
\end{center}\end{figure}
Figure \ref{fig:va1} shows the size
$1-h_1(\hz,\alpha)$ of the final set
of players $B$ as a function of the size of the initial seed $\alpha$
for various small values of $\lambda$. We see that for the smallest values of
$\lambda$, there is no discontinuity for the function $\alpha\mapsto
1-h_1(\hz,\alpha)$. In this case $\alpha_c(\lambda)$ is not defined.
We also see that there is a natural monotonicity in $\lambda$: as the
connectivity increases, the diffusion also increases.
However Figure \ref{fig:alphac} shows that for $\lambda\geq 2$, the
function $\lambda\mapsto \alpha_c(\lambda)$ becomes non-decreasing in
$\lambda$. Hence even if the connectivity increases, in order to
trigger a global cascade the size of the initial seed has to increase
too. The intuition for this phenomenon is now clear: in order to
trigger a global cascade, the seed has to overcome the local
stability of high-degree nodes which now dominates the effect of the
increase in connectivity.

\begin{figure}[htb]
\begin{center}
\includegraphics[width=5cm]{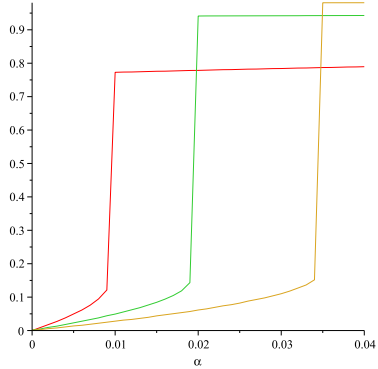}
\hspace{25pt}
\caption{Case $q>1/4$, high connectivity regime. Size of the final set of
  players $B$, $\alpha\mapsto 1-h_1(\hz,\alpha)$ as a function of the
  size of the initial seed $\alpha$ for different values of
  $\lambda=2,3,4$ (red, green, brown respectively).}
\label{fig:va2}
\end{center}\end{figure}

Figure \ref{fig:va2} shows the size
$1-h_1(\hz,\alpha)$ of the final set
of players $B$ as a function of the size of the initial seed $\alpha$
for various value of $\lambda$. Here we see that connectivity hurts
the start of the diffusion: for small value of $\alpha$, an increase
in $\lambda$ results in a lower size for the final set of players $B$!
However when $\alpha$ reaches the critical value $\alpha_c(\lambda)$,
then a global cascade occurs and its size is an increasing function of
$\lambda$. In other words, high connectivity inhibits the global
cascade but once it occurs, it facilitates its spread.

\subsection{Local Mean Field approximation}\label{sec:lmf}

In this subsection, we say that a player $A$ is inactive and a player $B$ is active.
We describe an approximation to the local
structure of the graph and present a heuristic argument which leads
easily to a prediction for the asymptotic probability of being active.
This heuristic allows us to recover the formulas for the functions $h,g,h_1$. 
Results from this section are not needed for the rigorous proofs given in Sections \ref{sec:thepi} and \ref{sec:casc} and are only presented to give some intuition about our technical formulas.
This branching process
approximation is standard in the random graphs literature \citep{dur} and is related to the generating function approach in physics \cite{wat02}. It is called Local Mean Field (LMF) in \citep{sig08,lel:sig09} and in the particular case of an epidemics with independent contagion on each edge, the LMF approximation was turned into a rigorous argument. For our model of diffusion with threshold, this seems unlikely to be straightforward and our proof will not use the LMF approximation.

The LMF model is characterized by its connectivity distribution
$\boldp$ defined in Condition \ref{cond}. We now construct a tree-indexed process.
Let $T$ be a Galton-Watson branching process \citep{atne} with a root
which has offspring distribution $p_r$ and all other nodes have
offspring distribution $p^*_r$ given by $p^*_{r-1} =
\frac{rp_r}{\lambda}$ for all $r\geq 1$. Recall that $p^*_{r-1}$
corresponds to the probability that an
edge points to a node with degree $r$, see \citep{dur}.
The tree $T$ describes the local structure of the graph $G$
(as $n$ tends to infinity): the exploration of the successive
neighborhoods of a given vertex is approximated by the branching
process $T$ as long as the exploration is local (typically restricted
to a finite neighborhood independent of $n$).

We denote by $\o$ the root of the tree and for a node $i$, we denote by
$\gen(i)\in \Nbold$
the generation of $i$, i.e. the length of the minimal path from $\o$
to $i$. Also we denote $i\to j$ if $i$ belongs to the children of
$j$, i.e. $\gen(i)=\gen(j)+1$ and $j$ is on the minimal path from
$\o$ to $i$.
For an edge $(i,j)$ with $i\to j$, we denote by $T_{i\to j}$ the
sub-tree of
$T$ with root $i$ obtained by the deletion of edge $(i,j)$ from $T$.

We now consider the diffusion model described in Section \ref{sec:ad} (with threshold for a vertex of degree $d$ given by $K(d)=qd$ and probability for a neighbor to be influential given by $\pi\in [0,1]$) on the tree $T$, where the
initial set of active nodes is given by a vector
$\sigma=(\sigma_i)_{i\in [n]}$: if $\sigma_i=1$ then vertex $i$ is in the initial set of active vertices, otherwise $\sigma_i=0$. In our model the $\sigma_i$'s are independent
Bernoulli random variables with parameter $\alpha_{d_i}$ where $d_i$
is the degree of node $i$ in the tree.
Thanks to
the tree structure, it is possible to compute the probability of being
active recursively as follows: for any node $i\neq \o$, let $Y_i=1$ if
node $i$ is active on the sub-graph $T_{i\to j}\cup\{(i,j)\}$ with
initial set of active nodes given by the restriction of $\sigma$ to
individuals in $T_{i\to j}$ and with individual $j$ held fix in the
inactive state. Then for any node $i\neq\o$ of degree $d_i$, $i$ becomes active if the number of active influential children exceeds $q d_i$. Hence, we get
\begin{eqnarray}
\label{eq:Y}Y_i=1-(1-\sigma_i)\ind\left(\sum_{\ell\to i}B_{\ell i}Y_\ell \leq q d_i \right),
\end{eqnarray}
where the $B_{\ell i}$'s are independent Bernoulli random variables with
parameter $\pi$.
Then the state of the root is given by
\begin{eqnarray}
\label{eq:X}X_{\o} =1-(1-\sigma_{\o})\ind\left(\sum_{i\to \o}B_{i\o}Y_i \leq qd_{\o} \right).
\end{eqnarray}
In order to compute the distribution of $X_{\o}$, we first solve the
Recursive Distributional Equation (RDE) associated to the
$Y_i$'s: thanks to the tree structure, the random variables $Y_\ell$ in
(\ref{eq:Y}) are i.i.d. and have the same distribution as $Y_i$.
Hence their distribution solve the RDE given by
\begin{eqnarray}
\label{eq:RDE}Y\ed 1-(1-\sigma(D^*+1))\ind\left(\sum_{i=1}^{D^*}B_iY_i \leq q(D^*+1) \right),
\end{eqnarray}
where for a given $d$, the random
variable $\sigma(d)$ is Bernoulli with parameter $\alpha_d$, $B_i$'s are independent Bernoulli
with parameter $\pi$, $D^*$ has distribution $p^*_r$, $Y$ and the
$Y_i$'s are i.i.d. copies (with unknown distribution).
To solve the RDE (\ref{eq:RDE}), we need to compute only the mean of the Bernoulli random variable $Y$. Hence taking expectation in (\ref{eq:RDE}) directly gives a fixed-point equation for this mean and the following lemma follows (its proof is deferred to Appendix \ref{sec:techlem}):

\begin{Lemma}\label{lem:lmf}
The solutions of the RDE (\ref{eq:RDE}) are Bernoulli random variable with mean $x$, where $x$ is any solution of the fixed point equation: $\lambda (1-x)(1-x\pi)=h(1-x;\balpha,\boldp, \pi)$. 
For any such solution, the random variable $X_{\o}$ defined in (\ref{eq:X}) is a Bernoulli random
variable with mean $1-h_1(1-x;\balpha,\boldp, \pi)$.
\end{Lemma}

By the change of variable $z=1-x$, we see that Lemma \ref{lem:lmf} is
consistent with Proposition \ref{prop:ad} and allow to recover the functions $h,g,h_1$.
However note that, the RDE (\ref{eq:RDE}) has in general several solutions one being the trivial one with $x=1$ corresponding to the trivial equilibrium where every node are active. The fact that this RDE has several solutions is a major difficulty in order to turn this LMF approximation into a rigorous argument.
Finally, the crucial point allowing previous computation is the fact that in recursion (\ref{eq:Y})
the $Y_i$ can be computed ``bottom-up'' so that the $Y_i$'s of a given generation (from the root) are independent.
The $Y_i$'s in (\ref{eq:X}) encode the
information that $i$ is activated by a node in the subtree of $T$
``below'' $i$ (and not by the root). If one considers a node in the
original graph and runs a directed contagion model on a local neighborhood of
this node where only 'directed' contagion toward this node are
allowed, then the state of the graph seen from this node is well
approximated by the $Y_i$'s.

\section{General model and main results}\label{sec:mod}

We first present the model
for the diffusion process and then our main results for the spread of
the diffusion.

\subsection{Diffusion: percolated threshold model}\label{sec:stm}

In this section, we describe the diffusion for any given finite graph $G$ with vertex set $[n]$. We still denote by $d_i$ the degree of node $i\in [n]$.
From now on, a vertex $i$ is either active or inactive.
In our model, the initial set of active nodes $S$ (the seed) will remain active
during the whole process of the diffusion (results for the non-monotonic model are derived in Section \ref{sec:inactive}).
We will consider {\bf single activation} where the seed contains only one vertex and
{\bf degree based activation} where the vertex $i$ is in the seed with
  probability $\alpha_{d_i}$, where $d_i$ is the degree of the
  vertex independently of the others. In other words, each node $i$ draws independently of each
  other a Bernoulli random variable $\sigma_i$ with parameter
  $\alpha_{d_i}$ and is considered as initially active if $\sigma_i=1$
  and not initially active otherwise.
In the case of degree based activation, we denote $\balpha=(\alpha_d)$
the parameters of this activation.
In particular, if $\alpha_{d} =\alpha$ for all $d$, then a fraction
$\alpha$ chosen uniformly at random among the population is activated
before the diffusion takes place.

{\bf Symmetric threshold model:}
We first present the symmetric threshold model which generalizes the
bootstrap percolation \citep{bp07}: a node becomes active when a certain
threshold fraction of neighbors are already active.
We allow the
threshold fraction to be a random variable with distribution depending on
the degree of the node and such that thresholds are independent among nodes.
Formally, we define for each node $i$ a sequence of random variables in $\Nbold$
denoted by $(K_i(d))_{d=1}^\infty$. The threshold associated to node
$i$ is $k_i=K_i(d_i)$ where $d_i$ is the degree of node $i$.
We assume that for any two vertices $i$ and $j$, the sequences $(K_i(d))_{d=1}^\infty$ and $(K_j(d))_{d=1}^\infty$ are independent and have the same law as a sequence denoted by $(K(d))_{d=1}^\infty$.
For $0\leq \ell\leq  s$, we denote $t_{s\ell}=
\Pb(K(s) = \ell)$ the probability distribution of the threshold for a
node of degree $s$.
For example in the model of \citep{mor} described in the
introduction, we take $K(d)=\lfloor qd \rfloor$ so that
$t_{s\ell}=\ind(\lfloor qs \rfloor =\ell)$.
We will use the notation $\boldk= (k_i)_1^n$ and $\boldt=(t_{s\ell})_{0\leq \ell\leq s}$ denotes the distribution of thresholds.

Now the progressive dynamic of the diffusion on the finite graph $G$ operates as follows: some
set of nodes $S$ starts out being active; all other
nodes are inactive.
Time operates in discrete steps $t=1,2,3,\dots$. At
a given time $t$, any inactive node $i$ becomes
active if its number of active neighbors is at least $K_i(d_i)+1$.
This in turn may cause other nodes to become active.
It is easy to see that the final set of active nodes (after $n$ time
steps if the network is of size $n$) only depends on the initial set
$S$ (and not on the order of the activations) and can be obtained
as follows: set $X_i=\ind(i\in S)$ for all $i$.
Then as long as there exists $i$ such that
$\sum_{j\sim i}X_j> K_i(d_i)$, set $X_i=1$, where $j\sim i$
  means that $i$ and $j$ share an edge in $G$.
When this algorithm finishes, the final state of node $i$ is
represented by $X_i$: $X_i=1$ if node $i$ is active and $X_i=0$
otherwise.

Note that we allow for the possibility $d_i=K_i(d_i)$ in which case,
node $i$ is never activated unless it belongs to the initial set $S$.
Note also that the condition $K_i(d_i)\geq 0$ is actually not
restrictive. If we wish to consider a model where $K_i(d_i)<0$ is
possible, we just need to modify the initial seed $S$ so as to put
node $i$ in $S$ if $K_i(d_i)<0$. Hence this model is equivalent to
ours if we increase $\alpha_{d_i}$ accordingly.

{\bf  Percolated threshold model:}
this model depends on a parameter $\pi\in [0,1]$ and a distribution
of random thresholds $(K(d))_{d\in \Nbold}$ given by
$\boldt=(t_{s\ell})_{\ell\leq s}$ as described above.
Given any graph $G$ and initial set $S$, we now proceed in two phases.
\begin{itemize}
\item bond percolation: randomly delete each edge with probability
  $1-\pi$ independently of all other edges. We denote the resulting
  random graph by $G_\pi$;
\item apply the symmetric threshold model with thresholds $K(d)$:
  set $X_i=\ind(i\in S)$ and then as long as there is $i$ such that
  $\sum_{j\sim_\pi i}X_j>K_i(d_i)$, set $X_i=1$, where $j\sim_\pi i$
  means that $i$ and $j$ share an edge in $G_\pi$ and $d_i$ is the
  degree of node $i$ in the original graph $G$.
\end{itemize}

Clearly if $\pi=1$, this is exactly the symmetric threshold model.
If in addition $K(d)=k$, then this model is known as bootstrap
percolation \citep{bp07}.
On the other hand if $\pi\in (0,1)$ and $K(d)=0$ for any $d$, then
this is the contact process with probability of activation $\pi$ on
each edge.
Note that the percolated threshold model is not equivalent to the
symmetric threshold model on the (bond) percolated graph since
threshold depends on the degree in the original graph (and not in the percolated graph).

\subsection{Diffusion with degree based activation}\label{sec:main}

Recall that for integers $s\geq 0$ and $0\leq r\leq s$, $b_{s r}$ denotes the
binomial probabilities
$b_{s r}(p) := \Pb(\Bi(s,p)=r) = {s \choose r} p^r(1-p)^{s-r}$.

For a graph $G$, let $v(G)$ and $e(G)$ denote the numbers of vertices
and edges in $G$ respectively.
In this section, we assume that the followings are given: a sequence of independent thresholds $\boldk$ drawn according to a distribution $\boldt$, a parameter $\pi\in [0,1]$ and an activation set drawn according to the distribution $\balpha$.
The subgraph of $G(n,\boldd)$ induced by the activated (resp. inactive) nodes at the end of the diffusion is denoted by $H$ (resp. $I$).
For $r\leq s$, we denote by $v_{sr}(I)$ the number of
vertices in $I$ with degree $s$ in $G$ and $r$ in $I$, i.e. the
number of vertices with degree $s$ in $G$ which are not activated and
with $r$ neighbors which are not activated either.
We denote by $v_s(H)$ the number of activated vertices of degree $s$
in $G$ (and with possibly lower degree in $H$).

We define the functions:
\begin{eqnarray*}
\label{eq:h}h(z;\boldp,\boldt,\balpha, \pi) &:=&
\sum_s (1-\alpha_s)p_s\sum_{\ell\leq s} t_{s\ell}\sum_{r\geq s-\ell}rb_{sr}(1-\pi+\pi z),\\
g(z;\boldp,\boldt,\balpha, \pi) &:=&
\lambda(\boldp)z(1-\pi+\pi z)-h(z;\boldp,\boldt,\balpha, \pi),\\
\label{eq:h1}h_1(z;\boldp,\boldt,\balpha, \pi) &:=& \sum_s (1-\alpha_s)p_s\sum_{\ell\leq s}t_{s\ell}\sum_{r\geq s-\ell} b_{sr}(1-\pi+\pi z),
\end{eqnarray*}
where $\lambda(\boldp)=\sum_{s}s p_{s}$. 
We define
\begin{eqnarray}
\label{eq:max}
\hz=\hz(\boldp,\boldt,\balpha, \pi) :=\max\left\{z\in [0,1]:\:
g(z;\boldp,\boldt,\balpha, \pi)= 0\right\}.
\end{eqnarray}
We refer to Appendix \ref{sec:max} for a justification of the use of the $\max$ in (\ref{eq:max}).

\begin{theorem}\label{th:epi}
Consider the random graph $G(n,\boldd)$ (or $G^*(n,\boldd)$) for a sequence
$\boldd=(d_i)_1^n$ satisfying Condition \ref{cond}.
We are given with a sequence of independent thresholds $\boldk$ drawn according to a distribution $\boldt$, a parameter $\pi\in [0,1]$ and an activation set drawn according to the distribution $\balpha$.
For the percolated threshold diffusion on the graph
$G(n,\boldd)$, we have:
if $\hz=0$, or
if $\hz\in (0,1]$, and further
$g(z;\balpha,\boldt,\boldp,\pi)<0$ for any $z$ in some interval $(\hz-\epsilon,\hz)$,
  then
\begin{eqnarray*}
v(H)/n&\stackrel{p}{\to}& 1-h_1(\hz;\boldp,\boldt,\balpha, \pi),\\
v_{sr}(I)/n&\stackrel{p}{\to}&\sum_{i+\ell\geq
  s-r}(1-\alpha_s)p_st_{s\ell}b_{sr}(\hz)b_{s-r,i}(1-\pi),\\
v_{s}(H)/n&\stackrel{p}{\to}&p_s-\sum_{k\geq
  s-\ell}(1-\alpha_s)p_st_{s\ell}b_{sk}(1-\pi+\pi\hz),\\
e(I)/n&\stackrel{p}{\to}&\left(\ind(\pi\neq 1)\frac{\hz}{2(1-\pi+\pi \hz)}+\ind(\pi=1)\frac{1}{2}\right) h(\hz;\boldp,\boldt,\balpha, \pi).
\end{eqnarray*}
If we condition the induced graph $I^*$ of inactive nodes in $G^*(n,\boldd)$ on its degree sequence $\boldd^{I^*}$ and let $n^{I^*}$ be the number of its vertices, then $I^*$ has the distribution of $G^*(n^{I^*},\boldd^{I^*} )$.
\end{theorem}

The proof of this theorem is given in Section \ref{sec:thepi}.
Note that Proposition \ref{prop:ad} follows easily.

\subsection{Diffusion with a single activation}

In this section, we look at the diffusion with one (or a small number $o(n)$ of) initial
active node(s) in $G(n,\boldd)$. 

For $u\in [1,n]$, let $\cC(u)$ (resp. $\cC(1,\dots,j)$) be the subgraph induced by the final
active nodes with initial active node $u$ (resp. initial active nodes
$1,\dots,k$). We also define $\cI(u)$ as the subgraph induced by the inactive nodes with initial active node $u$. The set of vertices of $\cC(u)$ and $\cI(u)$ is a partition of the vertices of the original graph.
To ease notation, we denote $h(z)=h(z;\boldp,\boldt,\bold{0},\pi)$, $g(z)=g(z;\boldp,\boldt,\bold{0},\pi)$ and
$h_1(z)=h_1(z;\boldp,\boldt,\bold{0},\pi)$. We define
\begin{eqnarray}
\label{eq:defxi}\xi :=\sup\left\{z\in [0,1):\:
g(z)= 0\right\}.
\end{eqnarray}

We also define
\begin{eqnarray*}
\bg(z) &=& (1-\pi+\pi z)\left( \lambda z -\sum_{s}sp_s(1-t_{s0})-\sum_{s}sp_st_{s0}(1-\pi+\pi z)^{s-1}\right),\\
\bh_1(z) &=& \sum_{s}p_st_{s0}(1-\pi+\pi z)^s+\sum_s p_s(1-t_{s0}),
\end{eqnarray*}
and $\bxi= \sup\left\{z\in [0,1):\: \bg(z)= 0\right\}$.

We call the following condition the cascade condition:
\begin{eqnarray}
\label{eq:casc}\pi \sum_r r(r-1) p_rt_{r0}>\sum_r rp_r,
\end{eqnarray}
which can be rewritten as $\pi\Eb[D(D-1)\ind(K(D)=0)]>\Eb[D]$ where $D$ is a random variable with distribution $\boldp$.

We denote by $P=P(n,\boldd,\boldk)$ the largest connected component of the graph $G(n,\boldd,\boldk)$ on which we apply a bond percolation with parameter $\pi$ (i.e. we remove each edge independently with probability $1-\pi$) and then apply a site percolation by removing all vertices with $k_i\geq 1$.
The vertices of the connected graph $P$ are called pivotal vertices: for any $u\in P$, we have $P\subset \cC(u)$.

\begin{theorem}\label{prop:cascade}
Consider the random graph $G(n,\boldd)$ (or $G^*(n,\boldd)$) for a sequence
$\boldd=(d_i)_1^n$ satisfying Conditions \ref{cond} and
\ref{cond2}.
We are given with a sequence of independent thresholds $\boldk$ drawn according to a distribution $\boldt$ and a parameter $\pi\in [0,1]$.
\begin{itemize}
\item[(i)] If the cascade condition (\ref{eq:casc}) is satisfied, then
\begin{eqnarray*}
\lim_n\frac{v(P)}{n} = 1-\bh_1(\bxi)>0.
\end{eqnarray*}
Moreover, for any $u\in P$, we have w.h.p.
\begin{eqnarray*}
\lim\inf_n \frac{v(\cC(u))}{n} = \lim\inf \frac{v(\cap_{u\in P}\cC(u))}{n}\geq 1-h_1(\xi) >0,
\end{eqnarray*}
where $\xi$ is defined by (\ref{eq:defxi}). Moreover if $\xi=0$ or $\xi$ is such that there exists $\epsilon>0$ with $g(z)<0$ for $z\in (\xi-\epsilon,\xi)$, then we have for any $u\in P$:
\begin{eqnarray*}
v(\cC(u))/n&\stackrel{p}{\to}& 1-h_1(\xi),\\
v_{sr}(\cI(u))/n&\stackrel{p}{\to}&\sum_{i+\ell\geq
  s-r}p_st_{s\ell}b_{sr}(\xi)b_{s-r,i}(1-\pi),\\
v_{s}(\cC(u))/n&\stackrel{p}{\to}&p_s-\sum_{k\geq
  s-\ell}p_st_{s\ell}b_{sk}(1-\pi+\pi\xi),\\
e(\cI(u))/n&\stackrel{p}{\to}&\left(\ind(\pi\neq 1)\frac{\xi}{2(1-\pi+\pi \xi)}+\ind(\pi=1)\frac{1}{2}\right) h(\xi).
\end{eqnarray*}
If we condition the induced graph $\cI^*(u)$ of inactive nodes in $G^*(n,\boldd)$ on its degree sequence $\boldd^{\cI^*(u)}$ and let $n^{\cI^*(u)}$ be the number of its vertices, then $\cI^*(u)$ has the distribution of $G^*(n^{\cI^*(u)},\boldd^{\cI^*(u)} )$.
\item[(ii)] If $\pi \sum_r r(r-1) r_rt_{r0}<\sum_r rp_r$, then for any $j=o(n)$,
  $v(\cC(1,\dots, j))=o_p(n)$.
\end{itemize}
\end{theorem}

A proof of this theorem is given in Section \ref{sec:casc}.
We end this section with
some remarks:
if the distribution of $K(d)$ does not depend on $d$, then the cascade condition becomes with $D$ a random variable with distribution $\boldp$:
\begin{eqnarray*}
\pi\Pb(K=0)>\frac{\Eb[D]}{\Eb[D(D-1)]}.
\end{eqnarray*}
In particular, if $K=0$ (i.e. the diffusion is a simple exploration process), then we find the well-known condition for
the existence of a 'giant component'. This corresponds to existing
results in the literature see in particular Theorem 3.9 in \citep{j08}
which extend the standard result of \citep{mr95}.
More generally, in the case $\pi\in [0,1]$ and $K=0$ (corresponding to the contact process), a simple computation shows that
\begin{eqnarray*}
h(z;\balpha,\boldp,\pi) &=&
(1-\alpha)(1-\pi+\pi z)\phi'_D(1-{\pi}+{\pi} z)\\
h_1(z;\balpha,\boldp,\pi) &=& (1-\alpha)\phi_D(1-{\pi}+{\pi} z),
\end{eqnarray*}
where $\phi_D(x)=\Eb[x^D]$ is the generating function of the asymptotic
degree distribution.
Applying Theorems \ref{th:epi} and \ref{prop:cascade} allow to obtain results for the contact process.
Similarly, the bootstrap percolation has been studied in random regular graphs \citep{bp07} and random graphs with given vertex degrees \citep{hamed}. The bootstrap percolation corresponds to the particular case of the percolated threshold model with $\pi=1$ and $K(d)=\theta\geq 0$ and our Theorems \ref{th:epi} and \ref{prop:cascade} allow to recover results for the size of the diffusion.
Finally, the case where $K(d)=qd$ and $\pi=1$ implies directly Proposition \ref{prop:cap}.

\section{Proof of Theorem \ref{th:epi}}\label{sec:thepi}

\subsection{Sketch of the proof}\label{sec:sketch}

It is well-known that it is often simpler to study the random
multigraph $G^*(n,\boldd)$ with given vertex sequence
$\boldd=(d_i)_1^n$ defined in Section \ref{sec:conf}.
We consider asymptotics as the number of vertices tends to infinity
and thus assume throughout the paper that we are given, for each $n$,
a sequence $\boldd=(d^{(n)}_i)_1^n$ with $\sum_i d^{(n)}_i$
even.
We may obtain $G(n,\boldd)$ by conditioning the multigraph
$G^*(n,\boldd)$ on being a simple graph. By \citep{j06}, we know that
the condition $\sum_i d_i^2=O(n)$ implies $\lim\inf \Pb(G^*(n,\boldd)
\mbox{ is simple})>0$. In this case, many results transfer immediately
from $G^*(n,\boldd)$ to $G(n,\boldd)$, for example, every result of
the type $\Pb(\cE_n)\to0$ for some events $\cE_n$, and thus every
result saying that some parameter converges in probability to some
non-random value. This includes every results in the present paper.
Henceforth, we will in this paper study the random multigraph $G^*(n,\boldd)$ and
in a last step (left to the reader) transfer the results to
$G(n,\boldd)$ by conditioning.

We run the dynamic of the diffusion of Section \ref{sec:stm} on a
general graph $G^*(n,\boldd)$ in order to compute the final size
of the diffusion in a similar way as in \citep{jl07}. The main point
here consists in coupling the construction of the graph with the
dynamic of the diffusion. The main difference with previous analysis consists in adding to each vertex a threshold with distribution $\boldt$ independently form the graph $G^*(n,\boldd)$ as described in Section \ref{sec:stm}.
The proof of Theorem \ref{th:epi} follows then easily see Section \ref{sec:spread}.
In order to prove Theorem \ref{prop:cascade}, we use the same idea of
coupling (in a similar spirit as in \citep{jlgiant} for the analysis of the
giant component) but we have to deal with an additional difficulty due
to the following lack of symmetry: if $\cC(u)$ is the final set of the
diffusion with only $u$ as initial active node, then for any $v\in
\cC(u)$, we do not have in general $\cC(u)= \cC(v)$.
We take care of this difficulty in Section \ref{sec:casc}.
In the next section, we present a preliminary lemma that will be used
in the proofs.

\subsection{A Lemma for death processes}\label{sec:lem}

A pure death process with rate 1 is a process that starts with some
number of balls whose lifetime are i.i.d. rate 1 exponentials. Each time a ball dies, it is removed.
Now consider $n$ bins with independent rate $1$ death processes.
To each bin, we attach a couple $(d_i,k_i)$ where $d_i$ is the number of balls
at time 0 in the bin and $k_i=K_i(d_i)$ is the threshold corresponding
to the bin. We now modify the death process as follows: all balls are
initially white.
For any living ball, when it dies, with probability $1-\pi$ and
independently of everything else, we color it green (an leave it in he bin), otherwise we remove it from the bin.
Let $W^{(n)}_j(t)$ and $G^{(n)}_j(t)$ denote the number of white
and green balls respectively in bin $j$ at time $t$, where
$j=1,\dots, n$ and $t\geq 0$.

Let 
$U^{(n)}_{sri,\ell}(t)$ be the
number of bins that have $s$ balls at time $0$ and $r$ white balls, $i$
green balls at time $t$ and threshold $\ell$, i.e. $U^{(n)}_{sri,\ell}(t) =
|\{j\in [n], \:W^{(n)}_j(t) =r,\: G^{(n)}_j(t)=i,\:d_j=s,\:k_j=
\ell\}|$.
In what follows we suppress the superscripts to lighten the
notation. The following lemma is an extension of Lemma 4.4 in
\citep{jl07}:
\begin{Lemma}\label{lem:death}
Consider the $n$ independent processes defined above and assume that
the sequence $\boldd=(d_i)_1^n$ satisfies Condition \ref{cond} where
$\boldp=(p_r)_{r=0}^\infty$ can be a defective probability distribution:
$\sum_r p_r\leq 1$. We assume that the thresholds $\{(K_i(d))_{d=1}^\infty\}_{i\in [n]}$ are independent in $i$ and have distribution $\boldt=(t_{s\ell})_{0\leq \ell\leq s}$ as described in Section \ref{sec:stm}.
Then, with the above notation, as $n\to \infty$,
\begin{eqnarray}
\label{eq:limsup}\sup_{t\geq 0} \sum_{\ell\leq s}\sum_{r=0}^s
\sum_{i=0}^{s-r}r\left|U_{sri,\ell}(t)/n-p_st_{s\ell}b_{s
  r}(e^{-t})b_{s-r,i}(1-\pi)\right| \pto 0.
\end{eqnarray}
\end{Lemma}
\begin{proof}
Let $n_{s\ell}= |\{i:\: d_i=s, k_i= \ell\}|$.
In particular
$\sum_{r,i}U_{sri,\ell}(0) =U_{ss0,\ell}(0) =
n_{s\ell}$.
First fix integers $s,\ell$ and $j$ with $0\leq  j\leq s$.
Consider the $n_{s\ell}$ bins that
start with $s$ balls and with threshold $\ell$. For $k=1,\dots,n_{s\ell}$, let $T_k$ be the time the $j$th ball is removed or recolored in the $k$th such bin.
Then $|\{k:\: T_k\leq t\}|=\sum_{r=0}^{s-j}\sum_{i=0}^{s-r}U_{sri,\ell}(t)$.
Moreover, for the $k$th bin, we have $\Pb(T_k\leq
t)=\sum_{r=0}^{s-j}b_{sr}(e^{-t})$.
Multiplying by $n_{s\ell}/n$ and using Glivenko-Cantelli theorem (see
e.g. Proposition 4.24 in \citep{kal}), we have
\begin{eqnarray*}
\sup_{t\geq 0}
\left|\frac{1}{n}\sum_{r=0}^{s-j}\sum_{i=0}^{s-r}U_{sri,\ell}(t)-\frac{n_{s\ell}}{n}\sum_{r=0}^{s-j}b_{sr}(e^{-t})\right|
\pto 0.
\end{eqnarray*}
Since $n_{s\ell}/n\to^p p_st_{s\ell}$ (by Condition \ref{cond}(i) and
the law of large numbers), we see that
\begin{eqnarray*}
\sup_{t\geq 0}\left|\frac{1}{n}\sum_{i=0}^{s-r}U_{sri,\ell}(t)-p_st_{s\ell}b_{sr}(e^{-t})\right|\pto 0
\end{eqnarray*}
The law of $G_k(t)$ given $d_k=s$ and $W_k(t)=r\leq s$
is a Binomial distribution with parameters
$s-r$ and $1-\pi$. Hence by the
law of large numbers, we have
\begin{eqnarray*}
\sup_{t\geq 0}
\left|\frac{U_{sri,\ell}(t)}{n}-p_st_{s\ell}b_{sr}(e^{-t})b_{s-r,i}(1-\pi)\right|\pto 0.
\end{eqnarray*}

Hence each term in
(\ref{eq:limsup}) tends to $0$ in probability.
The same holds for
any finite partial sum.
Let $\epsilon>0$ and $S$ be such that $\sum_{s=S}^\infty
sp_s<\epsilon$. By Condition \ref{cond}(iii), we have
$\sum_{s,\ell} sn_{s\ell}/n= \sum_s d_s/n \to \lambda = \sum_s sp_s$. Hence, also
$\sum_{s\geq S}\sum_\ell sn_{s\ell}/n\to\sum_{s=S}^\infty
sp_s<\epsilon$. So that for sufficiently large $n$, we get
$\sum_{s\geq S}\sum_\ell sn_{s\ell}/n<\epsilon$ and
\begin{eqnarray*}
\sup_{t\geq 0} \sum_{s\geq S}\sum_\ell\sum_{r=0}^s
\sum_{i=0}^{s-r}r\left|U_{sri,\ell}(t)/n-p_st_{s\ell}b_{s r}( e^{-t})b_{s-r,i}(1-\pi)\right|\\
\leq
\sup_{t\geq 0} \sum_{s\geq
  S}\sum_\ell\sum_{r=0}^s\sum_{i=0}^{s-r}r\left(U_{sri,\ell}(t)/n+p_st_{s\ell}b_{s
    r}( e^{-t})b_{s-r,i}(1-\pi)\right)\\
\leq \sum_{s\geq S}\sum_\ell s(n_{s\ell}/n+p_st_{s\ell})<2\epsilon,
\end{eqnarray*}
and the lemma follows.
\end{proof}

\subsection{Proof of the diffusion spread}\label{sec:spread}

Our proof of Theorem \ref{th:epi} is an
adaptation of the coupling argument in \citep{jl07}.
We start by analyzing the symmetric threshold model.
We can view the algorithm of Section \ref{sec:stm} as follows: start with
the graph $G$ and remove vertices from $S$, i.e. the initially active nodes.
As a result, if vertex
$i$ has not been removed, its degree has been lowered. We denote by $d_i^A$ the
degree of $i$ in the evolving graph.
Then iteratively remove
vertices $i$ such that $d_i^A< d_i-k_i$. All removed
vertices at the end of this procedure are active and all
vertices left are inactive.
It is easily seen that we obtain the same result by
removing edges where one endpoint satisfy $d_i^A< d_i-k_i$, until
no such edge remains, and finally removing all isolated vertices,
which correspond to active nodes.

Regard each edge as consisting of two half-edges, each half-edge
having one endpoint.
We introduce types of vertices. We set the type of vertices in the
seed $S$ to $B$.
Say that a vertex (not in $S$) is of type $A$ if $d_i^A\geq d_i-k_i$
and of type $B$ otherwise. In particular at the beginning of the
process, all vertices not in $S$ are of type $A$ since $d_i^A=d_i$ and
all vertices in $S$ are by definition of type $B$.
As the algorithm evolves, $d^A_i$ decreases so that some type $A$
vertices become of type $B$ during the execution of the algorithm.
Similarly, say that a half-edge is of type $A$ or $B$ when its
endpoint is.
As long as there is any half-edge of type $B$, choose one such
half-edge uniformly at random and remove the edge it belongs to. This
may change the other endpoint from $A$ to $B$ (by decreasing $d^A$)
and thus create new half-edges of type $B$. When there are no
half-edges of type $B$ left, we stop. Then the final set of active
nodes is the set of vertices of type $B$ (which are all isolated).

As in \citep{jl07}, we regard vertices as bins and half-edges as balls. At each step, we
remove first one random ball from the set of balls in $B$-bins and
then a random ball without restriction. We stop when there are no
non-empty $B$-bins. We thus alternately remove a random $B$-ball and a
random ball. We may just as well say that we first remove a random
$B$-ball. We then remove balls in pairs, first a random ball and then
a random $B$-ball, and stop with the random ball, leaving no $B$-ball
to remove.
We change the description a little by introducing colors. Initially
all balls are white, and we begin again by removing one random
$B$-ball. Subsequently, in each deletion step we first remove a random
white ball and then recolor a random white $B$-ball red; this is
repeated until no more white $B$-balls remain.

We now run this deletion process in continuous time such that, if
there are $j$ white balls remaining, then we wait an exponential time
with mean $1/j$ until the next pair of deletions. In other words, we
make deletions at rate $j$. This means that each white ball is deleted
with rate 1 and that, when we delete a white ball, we also color a
random white $B$-ball red. Let $A(t)$ and $B(t)$ denote the numbers of
white $A$-balls and white $B$-balls at time $t$, respectively, and
$A_1(t)$ denotes the number of $A$-bins at time $t$.
Since red balls are ignored, we may make a final change of rules, and
say that all balls are removed at rate $1$ and that, when a white ball
is removed, a random white $B$-ball is colored red; we stop when we
should recolor a white $B$-ball but there is no such ball.

Let $\tau$ be the stopping time of this process. First consider the
white balls only. There are no white $B$-balls left at $\tau$, so
$B(\tau)$ has reached zero. However, let us consider the last deletion
and recoloring step as completed by redefining $B(\tau)=-1$; we then
see that $\tau$ is characterized by $B(\tau)=-1$ and $B(t)\geq 0$ for
$0\leq t\leq \tau$. Moreover, the $A$-balls left at $\tau$ (which are
all white) are exactly the half-edges in the induced subgraph $I$ of inactive
nodes. Hence, the number of edges in this subgraph is
$\frac{1}{2}A(\tau)$, while the number of nodes not activated is
$A_1(\tau)$.



If we consider only the total number $A(t)+B(t)$  of white
balls in the bins, ignoring the types, the process is as follows: each
ball dies at rate $1$ and upon its death another ball is also
sacrificed. The process $A(t)+B(t)$ is a death process with rate $2$
(up to time $\tau$).
Consequently, by Lemma 4.3 of \citep{jl07} (or Lemma \ref{lem:death}
above), we have
\begin{eqnarray}
\label{eq:AB}\sup_{t\leq \tau}\left| A(t)+B(t) - n\lambda e^{-2t}\right| =o_p(n),
\end{eqnarray}
since Condition \ref{cond} (iii) implies $\sum_r r|\{i:\: d_i=r\}|/n\to \lambda$ .

Now if we consider the final version of the process restricted to
$A$-bins, it corresponds exactly to the death process studied in Section
\ref{sec:lem} above with $\pi=1$.
We need only to compute the initial condition for this process.
For a degree based activation, each vertex of degree $s$ is activated
(i.e. the corresponding bin becomes a $B$-bin) with probability
$\alpha_s$.
Hence by the law of large numbers, the number of $A$-bins with
initially $s$ balls and threshold $\ell$ is $n_{s\ell} =
(1-\alpha_s)p_st_{s\ell} n+o_p(n)$.
With the notation of Lemma \ref{lem:death}, we have
\begin{eqnarray*}
A(t) = \sum_{s\geq 1,\:r\geq s-\ell}rU_{sr0,\ell}(t),&\mbox{and, }&
A_1(t) = \sum_{s\geq 1,\:r\geq s-\ell}U_{sr0,\ell}(t),
\end{eqnarray*}
with the defective probability distribution $(1-\alpha_s)p_s$.
Hence by Lemma \ref{lem:death} we get (recall that $\pi=1$ here):
\begin{eqnarray*}
\sup_{t\leq \tau}\left|\frac{A(t)}{n}-\sum_{s\geq 1,r\geq
  s-\ell}r(1-\alpha_s)p_st_{s\ell}b_{sr}(e^{-t})\right|\pto 0.
\end{eqnarray*}
It is then easy to finish the proof as in \citep{jl07} for this
model (the complete argument is presented below for the more general percolated threshold model). In particular, it ends the proof of Theorem \ref{th:epi} for
the case $\pi=1$.

We now consider the percolated threshold model with $\pi<1$.
We modify the process as follows: for any white $A$-ball when it dies,
with probability $1-\pi$, we color it green instead of removing it. A
bin is of type $A$ if $r+i\geq s-\ell$, where $r$ is the number of
withe balls in the bin, $i$ the number of green balls (which did not
transmit the diffusion) and $s$ and $\ell$ are the initial degree and
threshold.
Let
$A(t)$ be the number of white $A$-balls. By Lemma \ref{lem:death}, we
now have:
\begin{eqnarray*}
\sup_{t\leq \tau}\left|\frac{A(t)}{n} -\sum_{s,r+i\geq
    s-\ell}r(1-\alpha_s)p_st_{s\ell}b_{sr}(e^{-t})b_{s-r,i}(1-\pi)\right|&\pto&0,\\
\sup_{t\leq \tau}\left|\frac{A_1(t)}{n} -\sum_{s,r+i\geq
    s-\ell}(1-\alpha_s)p_st_{s\ell}b_{sr}(e^{-t})b_{s-r,i}(1-\pi)\right|&\pto&0.
\end{eqnarray*}
In particular, we have thanks to Lemma \ref{lem:equiv} (in Appendix
\ref{sec:techlem}) for $t\leq \tau$,
\begin{eqnarray}
\label{eq:limtau}A(t)/n=\frac{e^{-t}}{1-\pi+\pi e^{-t}}h(e^{-t};\boldp,\boldt,\balpha,\pi)+o_p(n),\quad A_1(t)/n=h_1(e^{-t};\boldp,\boldt,\balpha,\pi)+o_p(n).
\end{eqnarray}
By looking at white balls (without taking types in consideration), we
see that Equation (\ref{eq:AB}) is still valid.
Hence, we have
\begin{eqnarray}
\label{eq:B}\sup_{t\leq
  \tau}\left|\frac{B(t)}{n}-\frac{e^{-t}}{1-\pi+\pi
  e^{-t}}g(e^{-t};\boldp,\boldt,\balpha, \pi)\right|\pto 0.
\end{eqnarray}

For simplicity we write $g(z)$ for $g(z;\boldp,\boldt,\balpha, \pi)$.
Assume now that $t_1>0$ is a constant independent of $n$ with
$t_1<-\ln \hz$ so that $\hz<1$ and $g(1)>0$.
Hence, we have $g(z)>0$ for $z\in (\hz,1]$ and thus $g(e^{-t})>0$ for
$t\leq t_1$. We can find some $c>0$ such that $g(e^{-t})\geq c$ for
$t\leq t_1$. But $B(\tau)=-1$, so if $\tau\leq t_1$ then
$\frac{e^{-\tau}}{1-\pi+\pi
  e^{-\tau}}g(e^{-\tau})-B(\tau)/n>c\frac{e^{-t_1}}{1-\pi+\pi
  e^{-t_1}}$ and from (\ref{eq:B}), we have $\Pb(\tau\leq t_1)\to 0$.
In case $\hz=0$, we may take any finite $t_1$ and find
$\tau\to^p\infty$ and (\ref{eq:limtau}) with $t\to \infty$, yields that
\begin{eqnarray*}
\lim_{t\to\infty}A(t)=o_p(n),\quad \lim_{t \to \infty}A_1(t)=nh_1(0;\boldp,\boldt,\balpha, \pi)+o_p(n).
\end{eqnarray*}
In case $\hz>0$, by the hypothesis we can find $t_2\in (-\ln \hz, -\ln
(\hz-\epsilon))$ such that $g(e^{-t_2})=-c<0$. If $\tau>t_2$, then
$B(t_2)\geq 0$ and thus $B(t_2)/n-\frac{e^{-t_2}}{1-\pi+\pi
  e^{-t_2}}g(e^{-t_2})\geq c\frac{e^{-t_2}}{1-\pi+\pi e^{-t_2}}$.
Hence by (\ref{eq:B}), we have
$\Pb(\tau\geq t_2)\to 0$.
Since we can choose $t_1$ and $t_2$ arbitrarily close to $-\ln \hz$,
we have $\tau\to^p -\ln \hz$ and (\ref{eq:limtau}) with $t=\tau$ yields that
\begin{eqnarray*}
A(\tau)=\frac{n\hz}{1-\pi+\pi \hz}h(\hz;\boldp,\boldt,\balpha, \pi)+o_p(n),\quad A_1(\tau)= nh_1(\hz;\boldp,\boldt,\balpha, \pi)+o_p(n).
\end{eqnarray*}
Note that $v(H) = n-A_1(\tau)$ and $e(I)=\frac{1}{2}A(\tau)$.
Hence we proved Theorem \ref{th:epi} for $v(H)/n$ and $e(I)/n$. The results for $v_{sr}(I)/n$ and $v_s(H)/n$
follows from the same argument, once we note that
\begin{eqnarray*}
v_{sr}(I) = \sum_{r+i\geq s-\ell}U_{sri,\ell}(\tau)\mbox{
  and, } v_s(H)=|\{j:d_j=s\}|-\sum_{r+i\geq s-\ell}U_{sri,\ell}(\tau).
\end{eqnarray*}
Finally, the statement concerning the distribution of the induced subgraph $I^*$ follows from the fact that this subgraph has not been explored when previous algorithm stops.

\section{Proof of Theorem \ref{prop:cascade}}\label{sec:casc}

We start this section with some simple calculations.
We define for $z\in [0,1]$,
\begin{eqnarray*}
a(z) &=&\sum_sp_s \sum_{\ell \leq s} t_{s\ell}\sum_{r\geq s-\ell}r b_{sr}(z),\\
h(z)&=&a(1-\pi+\pi z),\\
h_1(z)&=& \sum_s p_s \sum_{\ell \leq s} t_{s\ell}\sum_{r\geq s-\ell} b_{sr}(1-\pi+\pi z),\\
h_2(z)&=& \frac{z}{1-\pi+\pi z}h(z) \mbox{ if $\pi<1$ and } h_2(z)=
h(z) \mbox{ otherwise,}\\
g(z) &=& \lambda z(1-\pi+\pi z)-h(z),\\
f(z)&=&\frac{z}{1-\pi+\pi z}g(z) \mbox{ if $\pi<1$ and } f(z)=
g(z) \mbox{ otherwise.}
\end{eqnarray*}
For $s\geq 1$ and $\ell\geq 1$, we have
\begin{eqnarray*}
\frac{d}{dz}\sum_{r\geq s-\ell}r b_{sr}(z) &=& s^2z^{s-1}+\sum_{s-1\geq r\geq s-\ell}r(r-sz){s\choose r} z^{r-1}(1-z)^{s-r-1}\\
&=& s^2z^{s-1}-\sum_{s-1\geq r\geq s-\ell}r(sz-r)\frac{s}{z(s-r)}b_{s-1r}(z),
\end{eqnarray*}
so that for $z\in [0,1]$, we have $\left| \frac{d}{dz}\sum_{r\geq
    s-\ell}r b_{sr}(z)\right|\leq s^2+s^3$. Hence by
Condition \ref{cond2}, $a$ is differentiable on $[0,1]$ and we have
\begin{eqnarray*}
a'(z)=\sum_{s,\ell} p_st_{s\ell}\left(s^2z^{s-1}-\sum_{s-1\geq r\geq
    s-\ell}r(sz-r){s\choose r}z^{r-1}(1-z)^{s-1-r}\right).
\end{eqnarray*}
In particular, we have
\begin{eqnarray*}
a'(1) &=& \sum_{s,\ell}p_st_{s\ell}\left(s^2-\ind(\ell\geq
  1)s(s-1)\right)\\
&=& \sum_{s}p_{s}t_{s0}s(s-1)+\lambda,
\end{eqnarray*}
so that we have $g'(1) =\lambda(1+\pi)-\pi a'(1)$.
Note also that $f(1)=g(1)=0$ and $f'(1)=g'(1)$.

Consider now the case (ii) where $\pi\sum_{s}p_{s}t_{s0}s(s-1)<\lambda$, so that $g'(1)=f'(1)>0$.
The proof for an upper bound on $n^{-1}v(\cC(1,\dots,j))$ follows
easily from Theorem \ref{th:epi}. Take a parameter $\balpha =
(\alpha_d)_{d\in \Nbold}$ with $\alpha_d=\alpha>0$ for all
$d$. Clearly the final set of active nodes $H(\alpha)$ will be greater
than for any seed with size $o_p(n)$. Now when $\alpha$ goes to zero,
the fact that $f'(1)<0$ ensures that $\hz(\alpha)\to 1$ in Theorem
\ref{th:epi} so that $\lim_{\alpha\to 0}\lim_{n\to \infty}v(H(\alpha))/n=0$. Hence point
(ii) in Theorem \ref{prop:cascade} follows.

We now concentrate on the case where the cascade condition holds.
In particular we have $g'(1)<0$ so that $\xi$ defined in
(\ref{eq:defxi}) by $\xi = \sup\{z\in [0,1), g(z)=0\}$ is strictly
less than one and we have
\begin{eqnarray}\label{eq:pos}
f(z)>0, \:\forall z\in(\xi,1).
\end{eqnarray}
Also as soon as there exists $\epsilon>0$ such that $g(z)<0$ for $z\in (\xi-\epsilon, \xi)$, we can use the same argument as above. Since, we have in this
case $\hz(\alpha) \to \xi$ as $\alpha\to 0$, it gives an upper bound
that matches with the statement (i) of Theorem \ref{prop:cascade}.
In order to prove a lower bound, we follow the general approach of
\citep{jlgiant}.
We modify the algorithm defined in Section \ref{sec:spread} as
follows: the initial set $S$ now contains only one vertex, say $v_1$. When there is
no half-edge (or ball) of type $B$, we say that we make an exception and we
select a vertex (or a bin) of type $A$ 
uniformly at random among all vertices of type $A$. 
We declare its white half-edges of type
$B$ and remove its other half-edges, i.e. remove the green balls
contained in the corresponding bin if there are any.
Exceptions are done instantaneously.

For any set of nodes $v_1,\dots, v_k$, let $\cC(v_1,\dots, v_k)$ be
the subgraph induced by the final active nodes with initial active
nodes $v_1,\dots, v_k$.
If $S=\{v_1\}$, then clearly when the algorithm has exhausted the
half-edges of type $B$, it removed the subgraph $\cC(v_1)$ from
the graph and all edges with one endpoint in $\cC(v_1)$.
Then an exception is made by selecting a vertex say $v_2$ in $G
\backslash \cC(v_1)$.
Similarly when the algorithm exhausted the half-edges of type $B$, it
removed the subgraph $\cC(v_1,v_2)$ and all edges with one endpoint in
this set of vertices.
More generally, if $k$ exceptions are made consisting of selecting
nodes $v_1,\dots v_k$, then before the $k+1$-th exception is made (or at
termination of the algorithm if there is no more exception made), the
algorithm removed the subgraph $\cC(v_1,\dots, v_k)$ and all edges
with one endpoint in this set of vertices.

We use the same notation as in Section \ref{sec:spread}.
In particular, we still have:
\begin{eqnarray}
\label{eq:l2}\sup_{t\geq 0}\left| A(t)+B(t) - n\lambda e^{-2t}\right| =o_p(n).
\end{eqnarray}

We now ignore the effect of the exceptions by letting $\tA(t)$ 
be the number of white $A$ balls if no exceptions were made, i.e. assuming
$B(t)>0$ for all $t$.
If $\dm=\max_i d_i$ is the maximum degree of $G^*(n, \boldd)$, then
we have $\tA(0)=A(0)\in [n-\dm,n]$. By Condition \ref{cond} (iii),
$\dm=O(n^{1/2})$, and thus $n^{-1}\dm =o(n)$. Hence we can apply
results of previous section:
\begin{eqnarray}
\label{eq:tA}\sup_{t\geq 0}\left|\frac{\tA(t)}{n} -\sum_{s,r+i\geq
    s-\ell}r p_st_{s\ell}b_{sr}(e^{-t})b_{s-r,i}(1-\pi)\right|&\pto &0.
\end{eqnarray}

We now prove that:
\begin{eqnarray}
\label{eq:bd}0\leq \tA(t) -A(t)<\sup_{0\leq s\leq t}(\tA(s)-A(s)-B(s))+\dm.
\end{eqnarray}
The fact that $\tA(t)\geq A(t)$ is clear. Furthermore, $\tA(t)-A(t)$
increases only when exceptions are made.
If an exception is made at
time $t$, then the process $B$ reached zero and a vertex with $j$
white half-edges is selected so that $B(t)=j-1<\dm$. Hence we have
\begin{eqnarray*}
\tA(t)-A(t)<\tA(t)-A(t)-B(t)+\dm.
\end{eqnarray*}
Between exceptions, if $A$ decreases by one then so does $\tA$, hence
$\tA(t)-A(t)$ does not increase.
Consequently if $s$ was the last time before $t$ that an exception was
performed, then $\tA(t)-A(t)\leq \tA(s)-A(s)$ and (\ref{eq:bd})
follows.

Let $\tB(t) = A(t)+B(t)-\tA(t)$, then we have
\begin{eqnarray}
\label{eq:tB}\sup_{t\geq
  0}\left|\frac{\tB(t)}{n}-f(e^{-t})\right|\pto 0.
\end{eqnarray}
Equation (\ref{eq:bd}) can be written as
\begin{eqnarray}
\label{eq:bd1}0\leq \tA(t)-A(t)<-\inf_{s\leq t}\tB(s)+\dm
\end{eqnarray}

We first assume that $\xi$ given by (\ref{eq:defxi}) is such that
$\xi>0$ and  there exists $\epsilon>0$ such that $g(z)<0$ and hence $f(z)<0$ for $z\in (\xi-\epsilon,\xi)$.
Let $\tau=-\ln \xi$.
Then by (\ref{eq:pos}), we have $f(e^{-t})>0$ for $0<t<\tau$ so that
$\inf_{t\leq \tau} f(e^{-t})=f(1)=0$ and hence by (\ref{eq:tB}),
\begin{eqnarray}
\label{eq:tBinf}\inf_{t\leq \tau}n^{-1}\tB(t)\pto 0.
\end{eqnarray}
By Condition \ref{cond} (iii), $n^{-1}\dm
=o(n)$. Consequently, (\ref{eq:bd1}) yields
\begin{eqnarray}
\label{eq:tBtau}\sup_{t\leq \tau}n^{-1}|\tB(t)-B(t)|=\sup_{t\leq
  \tau}n^{-1}|\tA(t)-A(t)|\pto 0,
\end{eqnarray}
and thus by (\ref{eq:tB})
\begin{eqnarray}
\label{eq:B2}\sup_{t\leq \tau}\left|\frac{B(t)}{n}-f(e^{-t})\right|\pto 0.
\end{eqnarray}

By assumption, there exists
$\epsilon>0$ sufficiently small for $f(e^{-\tau-\epsilon})<0$. Since
$f(e^{-t})>0$ on the interval $[\epsilon, \tau-\epsilon]$,
(\ref{eq:B2}) implies that w.h.p. $B(t)$ remains positive on
$[\epsilon, \tau-\epsilon]$, and thus no exception is made during this
interval.

On the other hand, $f(e^{-\tau-\epsilon})<0$ and (\ref{eq:tB}) implies
$n^{-1}\tB(\tau+\epsilon)= f(e^{-\tau-\epsilon})+o_p(n)$, while
$B(\tau+\epsilon)\geq 0$. Thus with $\delta= -f(e^{-\tau-\epsilon})/2>0$, w.h.p.
\begin{eqnarray}\label{eq:tau}
\tA(\tau+\epsilon)-A(\tau+\epsilon)=
B(\tau+\epsilon)-\tB(\tau+\epsilon)\geq -\tB(\tau+\epsilon)>n\delta,
\end{eqnarray}
while (\ref{eq:tBtau}) yields $\tA(\tau)-A(\tau)<n\delta$ w.h.p.
Consequently, w.h.p.
$\tA(\tau+\epsilon)-A(\tau+\epsilon)>\tA(\tau)-A(\tau)$ and an
exception is performed between $\tau$ and $\tau+\epsilon$.

Let $T_1$ be the last time an exception was performed before $\tau/2$
and let $T_2$ be the next time it is performed. We have shown that for
any $\epsilon>0$, w.h.p. $0\leq T_1\leq \epsilon$ and $\tau-\epsilon\leq
T_2\leq \tau+\epsilon$.

\begin{Lemma}\label{lm:int}
Let $T^*_1$ and $T^*_2$ be two random times when an exception is
performed, with $T^*_1\leq T^*_2$, and assume that $T^*_1\pto t_1$
and $T^*_2\pto t_2$, where $0\leq t_1\leq t_2\leq \tau$. If $v(T^*_i)$
is the number of vertices removed by
the algorithm by time $T^*_i$, then
\begin{eqnarray}
\label{eq:lev}\frac{v(T^*_2)-v(T^*_1)}{n}&\stackrel{p}{\to}&h_1(e^{-t_1})-h_1(e^{-t_2}).
\end{eqnarray}
In particular, if $t_1=t_2$, then $v(T^*_2)-v(T^*_1)=o_p(n)$.
\end{Lemma}
\begin{proof}
By definition, we have:
\begin{eqnarray*}
v(T^*_2)-v(T^*_1) = A_1(T^*_1-)-A_1(T^*_2-).
\end{eqnarray*}
Since $T^*_2\stackrel{p}{\to}t_2\leq \tau$ and $f$ is continuous,
$\inf_{t\leq T^*_2}f(e^{-t})\stackrel{p}{\to}\inf_{t\leq
  t_2}f(e^{-t})=0$, and (\ref{eq:tB}) and (\ref{eq:bd1}) imply, in
analogy with (\ref{eq:tBinf}) and (\ref{eq:tBtau}),
$n^{-1}\inf_{t\leq T^*_2}\tB(t)\stackrel{p}{\to}0$ and
\begin{eqnarray}\label{eq:tA-A}
\sup_{t\leq T^*_2}n^{-1}|\tA(t)-A(t)|\stackrel{p}{\to}0.
\end{eqnarray}

Let $\tA_1(t)$ be the number of bins if no exceptions were made.
Then we clearly have $\tA_1(t)-A_1(t)\leq \tA(t)-A(t)$ since each time
an exception is made $\tA_1(t)-A_1(t)$ increases by one while
$\tA(t)-A(t)$ increases by more than one.
Hence (\ref{eq:lev}) follows from results in previous section.
\end{proof}

Let $\cC'$ (resp. $\cC''$) be the subgraph induced by the vertices removed by the algorithm
between $0$ and $T_1$ (resp. $T_2$). By Lemma \ref{lm:int}, we have
\begin{eqnarray}
\label{eq:C'} v(\cC')/n&\stackrel{p}{\to}&0\\
\label{eq:C''}v(\cC'')/n&\stackrel{p}{\to}&h_1(1)-h_1(\xi)=1-h_1(\xi).
\end{eqnarray}
Hence informally, the exception made at time $T_1$ triggers a large
cascade.

Now consider the case $\xi=0$. Note in particular that we have
$h(0)=0$ and since $h_1(z)\leq h(z)$, we also have $h_1(0)=0$.
Then with the same argument as above,
we have that
$B(t)$ remains positive on $[\epsilon,+\infty)$ and thus no exception
is made after a last exception made at time $T_1$ with $T_1\to^p
0$. Hence (\ref{eq:C'}) and (\ref{eq:C''}) follow with $\cC''$ being
the whole graph (except possibly $o_p(n)$ vertices), $h_1(0)=0$.

We now finish the proof of Theorem
\ref{prop:cascade} in the case where $\xi=0$ or there exists $\epsilon>0$ such that $g(z)<0$ for $z\in (\xi-\epsilon ,\xi)$.
First by \citep{j08} Theorems 3.5 and 3.9, the cascade condition implies that
$v(P)=\Omega_p(n)$. The result $\lim_n\frac{v(P)}{n} = 1-\bh_1(\bxi)>0$ could be derived from \citep{j08}. We give an alternative proof for this result at the end of this section.
For now, we denote $\gamma = 1-\bh_1(\bxi)>0$.
Coming back to the diffusion process analyzed above, we clearly have
$\cC'\cap P=\emptyset$. We now prove that $P\subset \cC''$.
This is clear in the case $\xi=0$. We now concentrate on the case
$\xi>0$.
First, let $T_3$ be the first time after $T_2$ when an exception is
made. Since $\tA(t)-A(t)$ increases by at most $d_{\max}=o(n)$ each
time an exception is made, we obtain from (\ref{eq:tA-A}):
\begin{eqnarray*}
\sup_{t\leq T_3} (\tA(t)-A(t))\leq \sup_{t\leq T_2} (\tA(t)-A(t))+d_{\max}=o_p(n).
\end{eqnarray*}
Hence similarly as in (\ref{eq:tau}), we have for every $\epsilon>0$,
w.h.p. $\tau+\epsilon > T_3$. Since also $T_3>T_2\to^p \tau$, it
follows that $T_3\to^p \tau$. If $\cC'''$ is the subgraph removed by
the algorithm between $T_2$ and $T_3$, then Lemma \ref{lm:int} implies
that $v(\cC''')=o_p(n)$.
Assume now that $\cC''\cap P=\emptyset$, then with probability at
least $\gamma>0$, the vertex chosen for the exception at $T_2$ belongs
to $P$ and then we have $\Pb(\cC''' \mbox{has more than $\gamma
  n$ vertices})\geq \gamma$, in contradiction with
$v(\cC''')=o_p(n)$. Hence we have $\cC''\cap P\neq \emptyset$ and then
$P\subset \cC''$ as claimed.

We clearly have for any $u\in P$:
\begin{eqnarray*}
v(\cC(P))=v(\cC(u))= v(\cap_{u\in P}\cC(u)).
\end{eqnarray*}
Hence we only need to prove that
$v(\cC(P))\geq 1-h_1(\xi)n+o_p(n)$.
To see this, attach to each vertex $i$ a random variable $U_i$ uniformly
distributed over $[0,1]$. Each time an exception has to be made, pick
the vertex among the remaining ones with minimal $U_i$ so that we do not change the algorithm described at the
beginning of the section.
From the analysis above, we see that all
exceptions made before $T_1$ are vertices not in $P$ and the exception made at time $T_1$ belongs to
$P$.
Now consider the graph $\tG$ obtained from the original graph where $\cC'$
has been removed but all other variables are the same as in the
original graph. Since $v(\cC')=o_p(n)$, this graph satisfies
Conditions \ref{cond} and \ref{cond2}. Hence previous analysis applies
and we have in addition that the first exception made by the algorithm
belongs to $P$ since a pivotal vertex in $G$ is also pivotal in $\tG$.
Hence the subgraph of $\tG$ removed between times $\tilde{T}_1=0$ and $\tilde{T}_2$ by the algorithm is exactly $\tilde{\cC}(P)$ in $\tG$. Since $\tG$
is a subgraph of $G$, we have $\tilde{\cC}(P)\subset \cC(P)$ in
the original graph. And the first claim in (i) follows from (\ref{eq:C''}) applied
to the graph $\tG$.
The second claim in (i) follows exactly as in the proof of Theorem \ref{th:epi} given in previous section.

Now consider the case where $\xi>0$ but for any $\epsilon>0$, there exists $z\in (\xi-\epsilon,\xi)$ such that $g(z)\geq0$ and hence $f(z)>0$.
The idea to get a lower bound is to let $\pi$ vary.
Since $\pi>0$, for any
$0<\pi'<\pi$, we see that by a standard coupling argument, for any given
initial seed all active nodes in the model with $\pi'$ will also be active in the
model with $\pi$. Hence the model with $\pi'$ provides a lower bound
for the number of active nodes in the model with $\pi$.
Now consider the function $g(z,\pi)=\lambda z(1-\pi(1-z))-a(1-\pi(1-z))$ as a function of $\pi$. We have
\begin{eqnarray*}
\frac{\partial g}{\partial \pi}(z,\pi) = (1-z)\left[ a'(1-\pi(1-z))-\lambda z\right]
\end{eqnarray*}
Since $\xi$ is a local minimum of $z\mapsto g(z,\pi)$ and $g$ is differentiable as a function of $z$, we have
\begin{eqnarray*}
\frac{\partial g}{\partial z}(\xi,\pi) = 0 \Leftrightarrow a'(1-\pi(1-\xi)) = \frac{\lambda}{\pi}(1-\pi(1-\xi))+\lambda \xi.
\end{eqnarray*}
Hence we have
$\frac{\partial g}{\partial \pi}(\xi,\pi) = \frac{\lambda}{\pi}(1-\xi)(1-\pi(1-\xi))>0$.
In particular for $\pi-\epsilon<\pi'<\pi$, we have $g(\xi,\pi')<g(\xi,\pi)=0$.
Let $\xi(\pi')=\sup\{z\in [0,1),\: g(z,\pi')=0\}$, then we have $\xi(\pi')>\xi$ for any $\pi'<\pi$ and $g(z,\pi')<0$ for $z\in (\xi,\xi(\pi'))$.
Moreover, we have $\xi(\pi')\to \xi(\pi)$ as $\pi'\to \pi$ and
previous argument is valid for the model with $\pi'$ as close as
desired from $\pi$ showing that $1-h_1(\xi(\pi'))\to 1-h_1(\xi)$ is a
lower bound for the fraction of active nodes in the model with $\pi$.

We finish this proof by computing the asymptotic for the size of $P$ using our previous analysis but for a modified threshold as done in \citep{wine}. We add a bar for quantities associated to this new model.
Namely, consider a modification of the original diffusion with threshold
$\bK_i(d_i)=(d_i+1)\ind(K_i(d_i)\geq 1)$. In words, a node $i$ becomes
active if one of its neighbor is active and $K_i(d_i)=0$ in the
original diffusion model. Clearly the nodes that become active in this model
need to have only one active neighbor in the original diffusion model
with parameter $K_j(d_j)$.
We denote by $\bcC(u)$ the subgraph of final active nodes with
initial active node $u$. Note that our algorithm is an exploration process of the components of the graph on which we apply a bond percolation with parameter $\pi$ and a site percolation by removing all vertices with $k_i\geq 1$.
In particular, the largest component explored is exactly $P$. Note that the computations made before are valid with the following functions:
\begin{eqnarray*}
\bg(z) &=& (1-\pi+\pi z) \left\{ \lambda z -\sum_s sp_s(1-t_{s0}) -\sum_s sp_st_{s0}(1-\pi+\pi z)^{s-1}\right\}\\
\bh_1(z)&=& \sum_s p_s(1-t_{s0}) +\sum_s p_st_{s0}(1-\pi+\pi z)^{s}.
\end{eqnarray*}
Hence we have $\bxi = \sup \{z\in [0,1],\: \bg(z)=0 \}$.
Note that if we denote $\phi(z)=\bg(z)(1-\pi+\pi z)^{-1}$, then we have
\begin{eqnarray*}
\phi'(z) &=& \lambda -\pi \sum_s s(s-1)p_st_{s0} (1-\pi+\pi z)^{s-2},\\
\phi''(z) &=& -\pi^2 \sum_{s}s(s-1)(s-2) p_st_{s0} (1-\pi+\pi z)^{s-3}.
\end{eqnarray*}
In particular, $\phi$ is concave on $(0,1]$ and strictly concave unless $p_s=0$ for $s\geq 3$.
Note also that $\phi'(1) = \lambda -\pi\sum_{s}s(s-1)p_st_{s0}$, so that under the cascade condition $\phi'(1)<0$ and $\phi$ is strictly concave. Hence, we have $\bxi<1$ and $\phi(x)>0$ for $x\in (\bxi,1)$ and if $\xi>0$, then $\phi(x)<0$ for $x<\bxi$. In particular, previous analysis allows to conclude that $\lim_n\frac{v(P)}{n} = 1-\bh_1(\bxi)>0$.

\section{Proof of Propositions \ref{prop:inactiveer} and \ref{prop:inactive}}\label{sec:inactive}

We start with the proof of Proposition \ref{prop:inactive}.
By Theorem \ref{prop:cascade}, when the cascade condition holds, the set of active vertices contain the set of pivotal vertices and hence has a giant component. We denote $\cI$ the induced subgraph of inactive vertices, in the pivotal equilibrium (i.e. in the final state when all pivotal nodes are initially active).
By Theorem \ref{prop:cascade}, we have for $v_r(\cI)$, the number of vertices in $\cI$ with degree $r$ in $\cI$:
\begin{eqnarray*}
\frac{v_r(\cI)}{n} \stackrel{p}{\to} \sum_{r\geq s-\lfloor qs \rfloor}p_sb_{sr}(\xi),\mbox{ where, } \xi=\xi(\lambda)=\max\left\{z<1,\: \lambda z^2 = \sum_s p_s\sum_{r\geq s-\lfloor qs \rfloor}rb_{sr}(z) \right\} .
\end{eqnarray*}
We denote $v_r(z)=\sum_{r\geq s-\lfloor qs \rfloor}p_sb_{sr}(z)$.
Thanks to the result on the distribution of $\cI$, there is a giant component of inactive vertices if
\begin{eqnarray*}
\sum_r r(r-1)v_r(\xi)>\sum_r rv_r(\xi)=\sum_s p_s\sum_{r\geq s-\lfloor qs \rfloor}rb_{sr}(\xi)=\lambda \xi^2,
\end{eqnarray*}
which can be rewritten as in Proposition \ref{prop:inactive}.

We now prove Proposition \ref{prop:inactiveer}.
We assume that $p_r=\frac{\lambda^r}{r!}e^{-\lambda}$, $\pi=1$ and $K(d)=\lfloor qd\rfloor$.
The function $\psi(\lambda)= e^{-\lambda}\sum_{r<q^{-1}} \frac{\lambda^{r-1}}{(r-2)!}$ is increasing for $\lambda\leq \lambda^*$ and then decreasing. We assume that $q$ is fixed such that the cascade condition holds. Then $\psi(\lambda^*)>1$.
Then $\lambda_i(q) = \sup \{\psi(\lambda)<1,\lambda<\lambda^*\}$ while $\lambda_s(q)= \inf\{\psi(\lambda)>1,\:\lambda>\lambda^*\}$.
We denote $$\zeta(\lambda)= \max\left\{z<1,\: \lambda z^2 = \sum_s p_s\sum_{r\geq s-\lfloor qs \rfloor}r(r-1)b_{sr}(z) \right\}.$$
Then, there is a giant component of inactive vertices if $\zeta(\lambda)<\xi(\lambda)$.
Both functions are non-increasing in $\lambda\in (\lambda_i(q),\lambda_s(q))$ and are intersecting only once in $(\lambda_i(q),\lambda_s(q))$.

\section{Conclusions}\label{sec:conc}

This paper analyzes the spread of new behaviors or technologies in large social networks. Our analysis is motivated by the two qualitative features of global cascades in social and economics systems: they occur rarely, but are large when they do.

In our simplest model, agents play a local interaction binary game where the underlying social network is modeled by a sparse random graph. Considering the deterministic best response dynamics, we compute the contagion threshold for this model. We find that when the social network is sufficiently sparse, the contagion is limited by the low connectivity of the network; when it is sufficiently dense, the contagion is limited by the stability of the high-degree nodes. This phenomenon explains why contagion is possible only in a given range of the global connectivity.

Most importantly, we identify the set of agents able to trigger a large cascade: the pivotal players, i.e. the largest component of players requiring a single neighbor to change strategy in order to follow the change. The notion of pivotal player is crucial in the context of random graphs. Cascades occur only when pivotal players represent a positive fraction of the population and in this case, any cascade will be triggered by such a pivotal player.
We found that these pivotal players exist only if the connectivity of the network is in a given range. At both ends of this range (i.e. for low and high-connectivity), the number of pivotal players is low. However in the high-connectivity case, we found that the system displays a robust-yet-fragile quality: while pivotal players are very rare, they trigger very large cascades. This feature makes global contagions exceptionally hard to anticipate. We also analyze possible equilibria of the game and in particular, we find conditions for the existence of equilibria with co-existent conventions. 

Motivated by social advertising, we also consider cases where the number of pivotal players is a negligible fraction of the population. In this case, contagion is still possible if the set of initial adopters is sufficiently large. We compute the final size of the contagion as a function of the fraction of the initial adopters. We find that the low and high-connectivity cases still have different features: in the first case, the global connectivity helps the spread of the contagion while in the second case, high connectivity inhibits the global contagion but once it occurs, it facilitates its spread.

 Finally, we analyze a general percolated threshold model for the diffusion allowing to give different weights to the (anonymous) neighbors. Our general analysis gives explicit formulas for the spread of the diffusion in terms of the initial condition, the degree sequence of the random graph, and the distribution of the thresholds. It is hoped that this analysis will stimulate applications of our results to other practical cases. One such possibility is the study of financial networks as started in \cite{gaka10} and \cite{bat09}.


\begin{thebibliography}{}

\bibitem[\protect\citeauthoryear{Amini}{Amini}{2010}]{hamed}
Amini, H. (2010).
\newblock Bootstrap percolation and diffusion in random graphs with given
  vertex degrees.
\newblock {\em Electron. J. Combin.\/}~{\em 17\/}(1), Research Paper 25, 20.

\bibitem[\protect\citeauthoryear{Athreya and Ney}{Athreya and Ney}{1972}]{atne}
Athreya, K.~B. and P.~E. Ney (1972).
\newblock {\em Branching processes}.
\newblock New York: Springer-Verlag.
\newblock Die Grundlehren der mathematischen Wissenschaften, Band 196.

\bibitem[\protect\citeauthoryear{Bailey}{Bailey}{1975}]{bailey}
Bailey, N. T.~J. (1975).
\newblock {\em The mathematical theory of infectious diseases and its
  applications\/} (Second ed.).
\newblock Hafner Press [Macmillan Publishing Co., Inc.]\ New York.

\bibitem[\protect\citeauthoryear{Balogh and Pittel}{Balogh and
  Pittel}{2007}]{bp07}
Balogh, J. and B.~G. Pittel (2007).
\newblock Bootstrap percolation on the random regular graph.
\newblock {\em Random Structures Algorithms\/}~{\em 30\/}(1-2), 257--286.

\bibitem[\protect\citeauthoryear{Battiston, Gatti, Gallegati, Greenwald, and
  Stiglitz}{Battiston et~al.}{2009}]{bat09}
Battiston, S., D.~Gatti, M.~Gallegati, B.~Greenwald, and J.~Stiglitz (2009).
\newblock {Liaisons Dangereuses: Increasing Connectivity, Risk Sharing, and
  Systemic Risk.}

\bibitem[\protect\citeauthoryear{Blume, Easley, Kleinberg, Kleinberg, and
  Tardos}{Blume et~al.}{2011}]{bekkt}
Blume, L., D.~Easley, J.~Kleinberg, R.~Kleinberg, and E.~Tardos (2011).
\newblock {Which Networks Are Least Susceptible to Cascading Failures?}
\newblock In {\em FOCS 2011}.

\bibitem[\protect\citeauthoryear{Blume}{Blume}{1993}]{bl93}
Blume, L.~E. (1993).
\newblock The statistical mechanics of strategic interaction.
\newblock {\em Games and Economic Behavior\/}~{\em 5\/}(3), 387--424.

\bibitem[\protect\citeauthoryear{Blume}{Blume}{1995}]{bl95}
Blume, L.~E. (1995).
\newblock The statistical mechanics of best-response strategy revision.
\newblock {\em Games Econom. Behav.\/}~{\em 11\/}(2), 111--145.
\newblock Evolutionary game theory in biology and economics.

\bibitem[\protect\citeauthoryear{Bollob{\'a}s}{Bollob{\'a}s}{2001}]{bol01}
Bollob{\'a}s, B. (2001).
\newblock {\em Random graphs\/} (Second ed.), Volume~73 of {\em Cambridge
  Studies in Advanced Mathematics}.
\newblock Cambridge: Cambridge University Press.

\bibitem[\protect\citeauthoryear{Carlson and Doyle}{Carlson and
  Doyle}{2002}]{cd02}
Carlson, J. and J.~Doyle (2002).
\newblock {Complexity and robustness}.
\newblock {\em Proceedings of the National Academy of Sciences of the United
  States of America\/}~{\em 99\/}(Suppl 1), 2538.

\bibitem[\protect\citeauthoryear{Coupechoux and Lelarge}{Coupechoux and
  Lelarge}{2011}]{coulel}
Coupechoux, E. and M.~Lelarge (2011).
\newblock {Impact of Clustering on Diffusions and Contagions in Random
  Networks}.
\newblock In {\em NetGCooP 2011}.

\bibitem[\protect\citeauthoryear{Durrett}{Durrett}{2007}]{dur}
Durrett, R. (2007).
\newblock {\em Random graph dynamics}.
\newblock Cambridge: Cambridge University Press.

\bibitem[\protect\citeauthoryear{Easley and Kleinberg}{Easley and
  Kleinberg}{2010}]{easley2010}
Easley, D. and J.~Kleinberg (2010).
\newblock {\em {Networks, Crowds, and Markets: Reasoning About a Highly
  Connected World}}.
\newblock Cambridge University Press.

\bibitem[\protect\citeauthoryear{Ellison}{Ellison}{1993}]{el93}
Ellison, G. (1993).
\newblock Learning, local interaction, and coordination.
\newblock {\em Econometrica\/}~{\em 61\/}(5), 1047--1071.

\bibitem[\protect\citeauthoryear{Ellison and Fudenberg}{Ellison and
  Fudenberg}{1995}]{ef95}
Ellison, G. and D.~Fudenberg (1995).
\newblock Word-of-mouth communication and social learning.
\newblock {\em The Quarterly Journal of Economics\/}~{\em 110\/}(1), 93--125.

\bibitem[\protect\citeauthoryear{Gai and Kapadia}{Gai and
  Kapadia}{2010}]{gaka10}
Gai, P. and S.~Kapadia (2010).
\newblock {Contagion in financial networks}.
\newblock {\em Proceedings of the Royal Society A: Mathematical, Physical and
  Engineering Science\/}~{\em 466\/}(2120), 2401.

\bibitem[\protect\citeauthoryear{Galeotti and Goyal}{Galeotti and
  Goyal}{2009}]{gg08}
Galeotti, A. and S.~Goyal (2009).
\newblock Influencing the influencers: a theory of strategic diffusion.
\newblock {\em RAND Journal of Economics\/}~{\em 40\/}(3), 509--532.

\bibitem[\protect\citeauthoryear{Galeotti, Goyal, Jackson, and
  Vega-Redondo}{Galeotti et~al.}{2010}]{ggjvy08}
Galeotti, A., S.~Goyal, M.~O. Jackson, and F.~Vega-Redondo (2010).
\newblock Network games.
\newblock {\em Rev. Econom. Stud.\/}~{\em 77\/}(1), 218--244.

\bibitem[\protect\citeauthoryear{Harsanyi and Selten}{Harsanyi and
  Selten}{1988}]{hs88}
Harsanyi, J.~C. and R.~Selten (1988).
\newblock {\em A General Theory of Equilibrium Selection in Games}, Volume~1 of
  {\em MIT Press Books}.
\newblock The MIT Press.

\bibitem[\protect\citeauthoryear{Jackson}{Jackson}{2008}]{netbook}
Jackson, M.~O. (2008).
\newblock {\em Social and Economic Networks}.
\newblock Princeton, NJ, USA: Princeton University Press.

\bibitem[\protect\citeauthoryear{Jackson and Yariv}{Jackson and
  Yariv}{2007}]{jy07}
Jackson, M.~O. and L.~Yariv (2007).
\newblock Diffusion of behavior and equilibrium properties in network games.
\newblock {\em The American Economic Review\/}~{\em 97\/}(2).

\bibitem[\protect\citeauthoryear{Janson}{Janson}{2009a}]{j08}
Janson, S. (2009a).
\newblock On percolation in random graphs with given vertex degrees.
\newblock {\em Electronic Journal of Probability\/}~{\em 14}, 86--118.

\bibitem[\protect\citeauthoryear{Janson}{Janson}{2009b}]{j06}
Janson, S. (2009b).
\newblock The probability that a random multigraph is simple.
\newblock {\em Combin. Probab. Comput.\/}~{\em 18\/}(1-2), 205--225.

\bibitem[\protect\citeauthoryear{Janson and Luczak}{Janson and
  Luczak}{2007}]{jl07}
Janson, S. and M.~J. Luczak (2007).
\newblock A simple solution to the {$k$}-core problem.
\newblock {\em Random Structures Algorithms\/}~{\em 30\/}(1-2), 50--62.

\bibitem[\protect\citeauthoryear{Janson and Luczak}{Janson and
  Luczak}{2009}]{jlgiant}
Janson, S. and M.~J. Luczak (2009).
\newblock A new approach to the giant component problem.
\newblock {\em Random Structures Algorithms\/}~{\em 34\/}(2), 197--216.

\bibitem[\protect\citeauthoryear{Janson, {\L}uczak, and Rucinski}{Janson
  et~al.}{2000}]{rg}
Janson, S., T.~{\L}uczak, and A.~Rucinski (2000).
\newblock {\em Random graphs}.
\newblock Wiley-Interscience Series in Discrete Mathematics and Optimization.
  Wiley-Interscience, New York.

\bibitem[\protect\citeauthoryear{Kallenberg}{Kallenberg}{2002}]{kal}
Kallenberg, O. (2002).
\newblock {\em Foundations of modern probability\/} (Second ed.).
\newblock Probability and its Applications (New York). New York:
  Springer-Verlag.

\bibitem[\protect\citeauthoryear{Kandori, Mailath, and Rob}{Kandori
  et~al.}{1993}]{kmr}
Kandori, M., G.~J. Mailath, and R.~Rob (1993).
\newblock Learning, mutation, and long run equilibria in games.
\newblock {\em Econometrica\/}~{\em 61\/}(1), 29--56.

\bibitem[\protect\citeauthoryear{Kempe, Kleinberg, and Tardos}{Kempe
  et~al.}{2003}]{kkt03}
Kempe, D., J.~Kleinberg, and E.~Tardos (2003).
\newblock Maximizing the spread of influence through a social network.
\newblock In {\em Proceedings of the ninth ACM SIGKDD international conference
  on Knowledge discovery and data mining}, KDD '03, New York, NY, USA, pp.\
  137--146. ACM.

\bibitem[\protect\citeauthoryear{Lelarge}{Lelarge}{2008}]{wine}
Lelarge, M. (2008).
\newblock Diffusion of innovations on random networks: Understanding the chasm.
\newblock In C.~H. Papadimitriou and S.~Zhang (Eds.), {\em WINE}, Volume 5385
  of {\em Lecture Notes in Computer Science}, pp.\  178--185. Springer.

\bibitem[\protect\citeauthoryear{Lelarge}{Lelarge}{2009}]{lel:sig09}
Lelarge, M. (2009).
\newblock Efficient control of epidemics over random networks.
\newblock In J.~R. Douceur, A.~G. Greenberg, T.~Bonald, and J.~Nieh (Eds.),
  {\em SIGMETRICS/Performance}, pp.\  1--12. ACM.

\bibitem[\protect\citeauthoryear{Lelarge and Bolot}{Lelarge and
  Bolot}{2008}]{sig08}
Lelarge, M. and J.~Bolot (2008).
\newblock Network externalities and the deployment of security features and
  protocols in the internet.
\newblock In Z.~Liu, V.~Misra, and P.~J. Shenoy (Eds.), {\em SIGMETRICS}, pp.\
  37--48. ACM.

\bibitem[\protect\citeauthoryear{L{\'o}pez-Pintado}{L{\'o}pez-Pintado}{2008}]{%
lopez}
L{\'o}pez-Pintado, D. (2008).
\newblock Diffusion in complex social networks.
\newblock {\em Games Econom. Behav.\/}~{\em 62\/}(2), 573--590.

\bibitem[\protect\citeauthoryear{Molloy and Reed}{Molloy and Reed}{1995}]{mr95}
Molloy, M. and B.~Reed (1995).
\newblock A critical point for random graphs with a given degree sequence.
\newblock {\em Random Structures Algorithms\/}~{\em 6\/}(2-3), 161--179.

\bibitem[\protect\citeauthoryear{Montanari and Saberi}{Montanari and
  Saberi}{2010}]{ms09}
Montanari, A. and A.~Saberi (2010).
\newblock {The spread of innovations in social networks}.
\newblock {\em Proceedings of the National Academy of Sciences\/}~{\em
  107\/}(47), 20196--20201.

\bibitem[\protect\citeauthoryear{Morris}{Morris}{2000}]{mor}
Morris, S. (2000).
\newblock Contagion.
\newblock {\em Rev. Econom. Stud.\/}~{\em 67\/}(1), 57--78.

\bibitem[\protect\citeauthoryear{Newman}{Newman}{2003}]{new03}
Newman, M. E.~J. (2003).
\newblock The structure and function of complex networks.
\newblock {\em SIAM Rev.\/}~{\em 45\/}(2), 167--256 (electronic).

\bibitem[\protect\citeauthoryear{Richardson and Domingos}{Richardson and
  Domingos}{2002}]{rd02}
Richardson, M. and P.~Domingos (2002).
\newblock Mining knowledge-sharing sites for viral marketing.
\newblock In {\em KDD '02: Proceedings of the eighth ACM SIGKDD international
  conference on Knowledge discovery and data mining}, New York, NY, USA, pp.\
  61--70. ACM.

\bibitem[\protect\citeauthoryear{Vega-Redondo}{Vega-Redondo}{2007}]{vega07}
Vega-Redondo, F. (2007).
\newblock {\em Complex social networks}, Volume~44 of {\em Econometric Society
  Monographs}.
\newblock Cambridge: Cambridge University Press.

\bibitem[\protect\citeauthoryear{Watts}{Watts}{2002}]{wat02}
Watts, D.~J. (2002).
\newblock A simple model of global cascades on random networks.
\newblock {\em Proc. Natl. Acad. Sci. USA\/}~{\em 99\/}(9), 5766--5771
  (electronic).

\bibitem[\protect\citeauthoryear{Young}{Young}{1993}]{you93}
Young, H.~P. (1993).
\newblock The evolution of conventions.
\newblock {\em Econometrica\/}~{\em 61\/}(1), 57--84.

\bibitem[\protect\citeauthoryear{Young}{Young}{2000}]{you00}
Young, H.~P. (2000).
\newblock The diffusion of innovations in social networks.
\newblock Economics Working Paper Archive 437, The Johns Hopkins
  University,Department of Economics.

\bibitem[\protect\citeauthoryear{Young}{Young}{2009}]{py09}
Young, H.~P. (2009).
\newblock Innovation diffusion in heterogeneous populations: Contagion, social
  influence, and social learning.
\newblock {\em American Economic Review\/}~{\em 99\/}(5), 1899--1924.

\end{thebibliography}

\setcounter{section}{0}
\renewcommand{\thesection}{\Alph{section}}

\section{Appendix}
\subsection{Technical lemma}\label{sec:techlem}

\begin{Lemma}\label{lem:equiv}
For any $x,\pi\in [0,1]$, and $k\geq0$ we have
\begin{eqnarray*}
\frac{x}{1-\pi+x\pi} \sum_{r\geq k} rb_{s,r}(1-\pi+x\pi) = \sum_{r+i\geq k}rb_{s,r}(x)b_{s-r,i}(1-\pi).
\end{eqnarray*}
\end{Lemma}
\begin{proof}
This follows from the following observations for $x>0$:
\begin{eqnarray}
\nonumber\sum_{r\geq k} \frac{r}{1-\pi+x\pi} b_{s,r}(1-\pi+x\pi) &=& (s-1)\sum_{r\geq k-1}b_{s-1,r}(1-\pi(1-x))\\
\nonumber&=& (s-1)\sum_{r\leq s-k}b_{s-1,r}(\pi(1-x))\\
\label{eq:tec}&=& (s-1) \Pb(\sum_{i=1}^{s-1} B_iY_i \leq s-k),
\end{eqnarray}
where the $B_{i}$'s and $Y_i$'s are independent Bernoulli random variables with
parameter $\pi$ and $1-x$ respectively.
Now we also have:
\begin{eqnarray*}
\sum_{r+i\geq k}\frac{r}{x}b_{s,r}(x)b_{s-r,i}(1-\pi) &=& (s-1)\sum_{r+i\geq k-1}b_{s-1,r}(x)b_{s-1-r,i}(1-\pi)\\
&=& (s-1) \Pb(\sum_{i=1}^{s-1}(1-Y_i) +\sum_{i=1}^{s-1}Y_i(1-B_i)\geq k-1),
\end{eqnarray*}
which corresponds exactly to (\ref{eq:tec}).
\end{proof}

\subsection{Properties of the functions $h$ and $g$}\label{sec:max}

We now justify the use of the $\max$ in (\ref{eq:max}).
Let $\sigma(d)$ be a Bernoulli random variable with
parameter $\alpha_d$.
Let for $0\leq x\leq 1$, $D_x$ be the thinning of $D$ obtained by taking $D$ points and then randomly and independently keeping each of them with probability $x$. Thus given $D=d$, $D_x\sim \Bi(d,x)$. With these notations, we have
\begin{eqnarray*}
h(z;\boldp,\boldt,\balpha, \pi)&=&\Eb\left[ D_{1-\pi+\pi z}(1-\sigma(D))\ind\left(D_{1-\pi+\pi
    z}\geq D-K(D)\right)\right]\\
h_1(z;\boldp,\boldt,\balpha, \pi)&=& \Pb\left(\sigma(D)=0,\: D_{1-\pi+\pi
    z}\geq D-K(D)\right),
\end{eqnarray*}
so that both $h$ and $h_1$ are non-decreasing in $z$ and non-increasing in $\pi$.

Note that if $\balpha=0$, then $h(1;\boldp,\boldt,0, \pi)=\lambda$ so that
$g(1;\boldp,\boldt,0, \pi)=0$ and $\hz=1$.
We now consider the case $\balpha\neq 0$, so that there exists $d\geq 1$ such that $\alpha_d>0$.
In this case, we have $g(1;\balpha,\boldp, \pi) = \lambda-\sum_{s}s(1-\alpha_s)p_s\geq \alpha_d p_d d > 0$.
The statement then follows from the fact that the only possible jumps for $z\mapsto g(z;\boldp,\boldt,\balpha, \pi)$ are downwards. More precisely, let $\hz=\sup\{z\in [0,1], g(z;\boldp,\boldt,\balpha, \pi)=0\}$. Since the function $h$ is non-decreasing in $z$, its set of discontinuity points is denumerable say $\{z_i\}_{i\in \Nbold}$ and $h$ admits a left and right limit at any point $z$ denoted by $h(z-)$ and $h(z+)$ respectively. If $\hz\in \{z_i\}$, then we have
$h(\hz-)\leq \lambda \hz(1-\pi+\pi\hz)\leq h(\hz+)$.
In particular, we have $g(z;\boldp,\boldt,\balpha, \pi)\leq 0$ for any $z>\hz$ which contradicts the fact that $g(1;\boldp,\boldt,\balpha, \pi) >0$. Hence the functions $h$ and $g$ are continuous at $\hz$ and the $\sup$ is attained and can be replaced by a $\max$.

\subsection{Proof of Lemma \ref{lem:lmf}}

Taking expectation in (\ref{eq:RDE}), we get with $x=\Eb[Y]$,
\begin{eqnarray*}
1-x &=& \sum_{s\geq 0}p^*_s(1-\alpha_{s+1})\Pb\left(\sum_{i=1}^{s}B_iY_i
  \leq q(s+1) \right)\\
&=&  (1-\alpha_1)\frac{p_1}{\lambda}+\sum_{s\geq 1} \frac{(s+1)p_{s+1}}{\lambda}(1-\alpha_{s+1})
\sum_{j\leq q(s+1)}b_{sj}(x\pi)\\
&=&  \frac{1}{\lambda}\left\{(1-\alpha_1)p_1+\sum_{s\geq 1; j\geq s-\lfloor qs\rfloor} (1-\alpha_{s+1})p_{s+1} (s+1)b_{sj}(1-x\pi)\right\}.
\end{eqnarray*}
Note that $(s+1)b_{sj}(1-x\pi) = \frac{j+1}{1-x\pi}b_{s+1j+1}(1-x\pi)$ for $s\geq 1$, so that
\begin{eqnarray*}
1-x&=& \frac{1}{\lambda}\left\{(1-\alpha_1)p_1+\sum_{s\geq 1; j\geq s-\lfloor qs\rfloor}
\frac{(j+1)}{1-x\pi}(1-\alpha_{s+1})p_{s+1}b_{s+1j+1}(1-x\pi)\right\}
\\
&=& \frac{1}{\lambda (1-x\pi)}\sum_{s\geq 1; j\geq s-\lfloor qs\rfloor} j(1-\alpha_s)p_s b_{sj}(1-x\pi).
\end{eqnarray*}
Hence, we get
\begin{eqnarray*}
(1-x\pi)(1-x)\lambda = \sum_{s\geq 1; j\geq s-\lfloor qs\rfloor} j(1-\alpha_s)p_{s} b_{sj}(1-x\pi),
\end{eqnarray*}
and the first part of the lemma follows.
Taking expectation in (\ref{eq:X}) gives:
$\Eb[X_{\o}] =1-\sum_{s;j\geq s-\ell}(1-\alpha_s)p_st_{s\ell}b_{sj}(1-x\pi)$,
and the second part of the lemma follows.

\end{document}